\documentclass[a4paper, reqno]{amsart}

\usepackage{enumerate, amsmath, amsfonts, amsthm, amssymb, hyperref,
	mathtools, xcolor, old-arrows, tikz, tikz-cd, array, bussproofs,
	xspace, xstring}
\usepackage{adjustbox}
\usepackage{afterpage}
\usepackage{stmaryrd}
\usepackage{pdflscape}
\usepackage{tabularray}
\usepackage{geometry}
\usepackage{caption}
\usepackage[foot]{amsaddr}
\usepackage{newtxmath, newtxtext}

\usetikzlibrary{intersections}
\tikzstyle{braid}=[thick]
\tikzset{arr/.style={circle,draw,inner sep=0.03cm}}
\tikzset{unit/.style={circle,fill,inner sep=0.05cm}}
\tikzset{empty/.style={inner sep=0pt, minimum size=0pt}}

\newtheorem{theorem}{Theorem}[section]
\newtheorem{proposition}[theorem]{Proposition}
\theoremstyle{definition}
\newtheorem{definition}[theorem]{Definition}
\newtheorem{example}[theorem]{Example}
\theoremstyle{remark}
\newtheorem{remark}[theorem]{Remark}

\newcommand{\mA}{$\mathcal{A}$\xspace}
\newcommand{\A}{\mathcal{A}}

\newcommand{\mC}{$\mathcal{C}$\xspace}
\newcommand{\C}{\mathcal{C}}
\newcommand{\mD}{$\mathcal{D}$\xspace}
\newcommand{\D}{\mathcal{D}}
\newcommand{\mM}{$\mathcal{M}$\xspace}
\newcommand{\M}{\mathcal{M}}
\newcommand{\mV}{$\mathcal{V}$\xspace}
\newcommand{\V}{\mathcal{V}}

\newcommand{\sk}{\lhd}
\newcommand{\id}{\mathsf{id}}

\newcommand{\x}{\times}
\newcommand{\ox}{\otimes}
\newcommand{\Vect}{\mathsf{Vect}}

\newcommand{\Set}{\mathsf{Set}}
\newcommand{\mSet}{$\mathsf{Set}$\xspace}

\newcommand{\sw}[1]{\otimes^{#1}}
\newcommand{\Lan}{\mathsf{Lan}}
\newcommand{\Ran}{\mathsf{Ran}}

\newcommand{\inv}[1]{\big(#1\big)^{-1}}
\newcommand{\inhom}[1]{\llbracket #1 \rrbracket}
\newcommand{\uv}[1]{``#1''}

\newcommand{\smc}[1]{
	\IfEqCase{#1}{
		{}{skew monoidal\xspace}
		{l}{left skew monoidal\xspace}
		{r}{right skew monoidal\xspace}
	}[\PackageError{pret}{Undefined option to pret: #1}{}]%
}
\newcommand{\skmon}{skew monoidal\xspace}
\newcommand{\cat}{category\xspace}

\newsavebox\tikzcdbox
\newenvironment{tikzcdscale}{%
	\begin{lrbox}{\tikzcdbox}%
		\begin{tikzcd}%
		}{%
		\end{tikzcd}%
	\end{lrbox}%
	\resizebox{\linewidth}{!}{\usebox\tikzcdbox}%
}

\makeatletter
\newcommand*{\defeq}{\mathrel{\rlap{%
			\raisebox{0.2ex}{$\m@th\cdot$}}%
		\raisebox{-0.3ex}{$\m@th\cdot$}}%
	=}
\makeatother

\title{Skew monoidal structures on actegories}
\author[Pavla Procházková]{Pavla Procházková}
\address{Department of Mathematics and Statistics, Faculty of Science, Masaryk University, Brno, Czech Republic}
\date{}

\calclayout

\begin{document}
	
	\maketitle
	
	\begin{abstract}
		We present a construction of skew monoidal structures from strong actions. We prove that the existence of a certain adjoint allows one to equip the actegory with a skew monoidal structure and that this adjunction becomes monoidal.
		This construction provides a unifying framework for the description of several examples of skew monoidal categories. We also demonstrate how braidings on the original monoidal category of a given action induce braidings on the resulting skew monoidal structure on the actegory and describe sufficient conditions for closedness of the resulting skew monoidal structure.
	\end{abstract}
	
	\section{Introduction}
	
	Skew monoidal categories are a generalization of monoidal categories, where we remove the requirement that the constraint morphisms be invertible and specify certain orientations of those. An example of a structure of this kind appeared first, at least to the knowledge of the author, in the work of Altenkirch, Chapman and Uustalu in 2010 \cite{ACU10} (and later in a reworked version \cite{ACU15}). Their paper concerns relative monads, which can be regarded as a generalization of monads from endofunctors to functors of more general form $\C \rightarrow \D$. For a fixed functor $J: \C \rightarrow \D$, we speak of relative monads on the functor $J$ or simply $J$-relative monads. As monads can be regarded as monoids in the category of endofunctors, so can $J$-relative monads be regarded as monoids in a skew monoidal structure on the functor category $[\C,\D]$ with tensor $G \sk F = \Lan_JG \circ F$ (assuming that the left Kan extensions exist).
	
	Another class of examples which motivated introduction of these structures comes from the study of bialgebras, or more generally bialgebroids, for which the base ring need not be commutative. Szlachányi observed that bialgebroids give rise to these generalized monoidal structures and introduced the name skew monoidal category in \cite{Szl12}. In fact, he shows that bialgebroids can be fully characterized by closed skew monoidal structures. In the case of bialgebras, the tensor is defined as $X \ox B \ox Y$, where $B$ is the bialgebra\footnote{The tensor induced by bialgebroids is similar, just requiring more subtlety.}.
	
	The main focus of this paper is to study skew monoidal structures arising on actegories. The term actegory refers to a category \mA, which has some (skew) monoidal category \mV acting on it.
	We prove that given a strong action $(-) * (-) : \V \x \A \rightarrow \A$, the existence of an adjoint to the functor $(-) * J : \V \rightarrow \A$, obtained by fixing an object $J$ in \mA, suffices to equip the actegory with a skew monoidal structure.
	
	This result provides a unifying framework for many of the existing examples of skew monoidal categories. We see that the structure on functor categories defined via Kan extensions as well as all structures induced by bialgebras can be seen to arise in this way. Furthermore, it allows for a more conceptual treatment of several notions on skew monoidal categories, including braidings and closedness.
	
	The paper is structured in the following way. In Section (\ref{Prelim}), we recall the definition of a skew monoidal category and introduce strong actions of categories. Section (\ref{MainResult}) contains the statement and proof of the main result of the paper, which provides a construction of a skew monoidal structure on an actegory given that a certain adjunction exists.
	In Section (\ref{MonoidalitySection}), we show that this adjunction is monoidal with respect to the (skew) monoidal structure of the acting category \mV and the newly defined tensor on the actegory \mA. 
	
	Section (\ref{Examples}) is dedicated to describing examples of skew monoidal categories, which can be seen as arising from Theorem (\ref{TheThm}). We give an account of existing examples of skew monoidal categories, which fit under the framework of (\ref{TheThm}) -- these comprise the examples on functor categories of \cite{ACU15} and examples of monoidal tensors warped by an oplax monoidal monad (or monoidal comonad), which include examples arising from bialgebras. 
	
	In Section (\ref{Braiding}), we investigate braidings on the skew monoidal structure induced by actions. Starting with an action of a braided monoidal category, we prove that there is an induced braiding on the skew monoidal structure in the sense of Bourke and Lack \cite{BL20}. These authors define braidings on a left skew monoidal category as isomorphisms of form $(P \sk A) \sk B \rightarrow (P \sk B) \sk A$. One can also consider isomorphisms $A \sk (B \sk P) \rightarrow B \sk (A \sk P)$. In fact, considering different variations of Theorem (\ref{TheThm}), we can obtain either of those versions of a braiding on both left and right skew monoidal categories.
	
	In Section (\ref{Closedness}), we explore sufficient conditions for the skew monoidal structures induced by actions to be left or right closed.
	
	\textbf{Related work.} The idea that certain (skew) monoidal actions give rise to skew monoidal structures on the actegory has already been explored by Szlachányi in \cite{Szl17}. That paper treats cases where the action $*$ has the property that for any object $A$ in the actegory \mA, the induced functor $(-)*A$ has a right adjoint. We call these actegories closed. Proposition 6.1 of [\textit{ibid.}] 
	presents a construction of a skew monoidal structure from the data of a closed actegory, which is a result closely related to the main theorem of this thesis. However, our Theorem (\ref{TheThm}) is more general. The connection with Szlachányi's result is discussed in more detail in Remark (\ref{resSzl})
	
	The main theorem in this paper also seems to be closely related to a result presented at the CT in 2018 by Philip Saville as a part of a joint work with Marcelo Fiore. In their setting, one assumes a skew action $*: \A \x \V \rightarrow \A$ and a strong adjunction $F \dashv U: \V \rightarrow \A$. Then, $A * UB$ defines a left skew monoidal tensor on \mA with unit $FI$. 
	
	\textbf{Acknowledgements.} This paper is based on the author's master's thesis, which was supervised by John Bourke. The author acknowledges the support of Masaryk University under the grant MUNI/A/1569/2024. The author would like to thank her supervisor, the thesis opponent Soichiro Fujii, and the members of the Logic and Semantics Group at Tallinn University of Technology for all their helpful comments and suggestions. The author also gratefully acknowledges the valuable input of the anonymous referee.

	\section{Preliminary definitions}\label{Prelim}
	
	\subsection{Skew monoidal categories}
	
	\begin{definition}
		A \textit{left skew monoidal category} is given by a category \mA together with a product $\ox: \A \times \A \rightarrow \A$, unit $I \in \A$ and three natural families of morphisms
		\begin{align*}
			(A \ox B) \ox C &\xrightarrow{\gamma_{A,B,C}} A \ox (B \ox C)
			\tag{associator} \\
			I \ox A&\xrightarrow{\lambda_A} A  \tag{left unitor}\\
			A &\xrightarrow{\rho_A} A \ox I \tag{right unitor}
		\end{align*}
		satisfying the following five axioms:
		\[\begin{tikzcd} \label{LSkM1} \tag{LSkM1}
			& {(A \ox (B \ox C)) \ox D} \\
			{((A \ox B) \ox C) \ox D} && {A \ox ((B \ox C) \ox D)} \\
			\\
			{(A \ox B) \ox (C \ox D)} && {A \ox (B \ox (C \ox D))}
			\arrow["{\gamma_{A,B \ox C, D}}", from=1-2, to=2-3]
			\arrow["{\gamma_{A,B,C} \ox D}", from=2-1, to=1-2]
			\arrow["{\gamma_{A \ox B,C,D}}"', from=2-1, to=4-1]
			\arrow["{A \ox \gamma_{B,C,D}}", from=2-3, to=4-3]
			\arrow["{\gamma_{A,B,C \ox D}}"', from=4-1, to=4-3]
		\end{tikzcd}\]
		\[\begin{tikzcd} \label{LSkM2} \tag{LSkM2}
			& {A \ox B} \\
			{(A \ox B) \ox I} && {A \ox (B \ox I)}
			\arrow["{\rho_{A \ox B}}"', from=1-2, to=2-1]
			\arrow["{A \ox \rho_B}", from=1-2, to=2-3]
			\arrow["{\gamma_{A,B,I}}"', from=2-1, to=2-3]
		\end{tikzcd}\]
		\[\begin{tikzcd} \label{LSkM3} \tag{LSkM3}
			{(I \ox A) \ox B} && {I \ox (A \ox B)} \\
			& {A \ox B} 
			\arrow["{\gamma_{I,A,B}}", from=1-1, to=1-3]
			\arrow["{\lambda_A \ox B}"', from=1-1, to=2-2]
			\arrow["{\lambda_{A \ox B}}", from=1-3, to=2-2]
		\end{tikzcd}\]
		\[\begin{tikzcd} \label{LSkM4} \tag{LSkM4}
			& {I \ox I} \\
			I && I
			\arrow["{\lambda_I}", from=1-2, to=2-3]
			\arrow["{\rho_I}", from=2-1, to=1-2]
			\arrow[equal, from=2-1, to=2-3]
		\end{tikzcd}\]
		\[\begin{tikzcd} \label{LSkM5} \tag{LSkM5}
			{(A \ox I) \ox B} && {A \ox (I \ox B)} \\
			{A \ox B}         && {A \ox B}
			\arrow["{\gamma_{A,I,B}}", from=1-1, to=1-3]
			\arrow["{\rho_A \ox B}", from=2-1, to=1-1]
			\arrow["{A \ox \lambda_B}", from=1-3, to=2-3]
			\arrow[equal, from=2-1, to=2-3]
		\end{tikzcd}\]
	\end{definition}
	
	By reversing all arrows in the definition above, we obtain the definition of a\smc{r} \cat.
	
	\begin{remark}
		When generalizing the definition of a monoidal category by removing invertibility, one could also consider different orientations of the constraint morphisms and adjust the five axioms accordingly.
		
		The particular orientations of left/right skew monoidal categories provide several nice properties. These structures are self dual in the sense that if $(\A, \ox, J)$ is a left skew monoidal category, then $(\A^{\mathsf{op}}, \ox^{\mathsf{rev}}, J)$ with $A \ox^{\mathsf{rev}} B \defeq B \ox A$ is also a\smc{l} \cat.
		
		In a \skmon \cat, tensoring with the unit constitutes a co/monad. Namely in a\smc{l} \cat $(\A, \ox, J)$, the functor $(-) \ox J$ is a monad on \mA. By self-duality, the functor $J \ox (-)$ is then a comonad.
	\end{remark}
	
	\subsection{Strong actions of categories}
	
	Situations in which we have a monoidal category acting on another category in a \uv{monoidal way} (namely with invertible constraints) have been studied for a long time. A comprehensive treatment can be found for instance in \cite{CG22}. The notion of a skew monoidal category suggest the study of actions of a skew monoidal category which are skew in the sense that the constraints are generally not isomorphisms. This setting has been explored for instance by Szlachányi in \cite{Szl17}.
	
	Actions generally considered in this paper are in some sense a hybrid of these two notions. We will not make use of the invertibility of the constraint morphisms of the acting category, which allows us to consider skew monoidal categories. However, we will exploit the invertibility of the comparison morphisms of the action. This leads to the following notion, which we call a strong action.

	\begin{definition}\label{defAct}
		Let $(\V, \ox, I, a, \ell, r)$ be a \smc{l} category and \mA an ordinary category. A \textit{strong (left) action} of \mV on \mA is defined as a bifunctor
		\begin{align*}
			\V \times \A & \rightarrow \A \\
			(X, A)  & \mapsto X * A
		\end{align*}
		together with two natural families of isomorphisms
		\begin{align*}
			(X \ox Y) * A &\xrightarrow{m^{X,Y}_A} X * (Y * A) \tag{multiplicator}\\
			I * A &\xrightarrow{u_A} A \tag{unitor}
		\end{align*}
		such that \(*\) is compatible with the skew monoidal structure of \mV, i.e. the following coherence axioms are satisfied
		\[\begin{tikzcd} \label{LAct1} \tag{LAct1}
			& {(X \ox (Y \ox Z))*A} \\
			{((X \ox Y)\ox Z)*A} && {X * ((Y \ox Z)*A)}  \\
			\\
			{(X \ox Y) * (Z*A)} && {X*(Y*(Z*A))}
			\arrow["{m^{X,Y \ox Z}_{A}}", from=1-2, to=2-3]
			\arrow["{a_{X,Y,Z}*A}", from=2-1, to=1-2]
			\arrow["{m^{X\ox Y, Z}_{A}}"', from=2-1, to=4-1]
			\arrow["{X*m^{Y , Z}_{A}}", from=2-3, to=4-3]
			\arrow["{m^{X,Y}_{Z*A}}"', from=4-1, to=4-3]
		\end{tikzcd}\]
		\[\begin{tikzcd} \label{LAct2} \tag{LAct2}
			{(I \ox X)*A} && {I*(X*A)} \\
			& {X*A} 
			\arrow["{m^{I,X}_A}", from=1-1, to=1-3]
			\arrow["{\ell_X*A}"', from=1-1, to=2-2]
			\arrow["{u_{X*A}}", from=1-3, to=2-2]
		\end{tikzcd}\]
		\[\begin{tikzcd} \label{LAct3} \tag{LAct3}
			{X*A}         && {X*A}          \\
			{(X \ox I)*A} && {X*(I*A)}
			\arrow["{r_X*A}"', from=1-1, to=2-1]
			\arrow["{X*u_A}"', from=2-3, to=1-3]
			\arrow["{m^{X,I}_A}"', from=2-1, to=2-3]
			\arrow[equal, from=1-1, to=1-3]
		\end{tikzcd}\]
		We say that $(\A,*,m,u)$ is a strong (left) \mV-action, or simply that \mA is a strong (left) \mV-actegory.
	\end{definition}
	
	\begin{remark}
		Similarly, we can define a \textit{strong left action} of a right skew monoidal category -- here, we can consider multiplicator of form $X*(Y*A) \rightarrow (X \ox Y)*A$ and unitor of form $A \rightarrow I*A$ and define the corresponding action axioms by reversing the arrows in (\ref{LAct1}), (\ref{LAct2}) and (\ref{LAct3}).
		
		We will also consider \textit{strong right actions} of a left skew monoidal category \mV on \mA, which comprise a bifunctor
		\begin{align*}
			\A \times \V & \rightarrow \A \\
			(A, X)  & \mapsto A * X
		\end{align*}
		and for any $X,Y$ in \mV and $A$ in \mA a multiplicator morphism $(A*X)*Y \rightarrow A*(X \ox Y)$ and a unitor morphism $A \rightarrow A*I$ and three axioms analogous to (\ref{LAct1}), (\ref{LAct2}) and (\ref{LAct3}).
		
		Analogously, we may also define a strong right action of a right skew monoidal category.
	\end{remark}
	
	\begin{remark}
		If we take \mV to be monoidal in (\ref{defAct}), we recover the usual notion of a \textit{monoidal action} of \mV on \mA as defined for instance in \cite{CG22}.
	\end{remark}
	
	\section{Skew monoidal structures arising from actions}\label{MainResult}

	Given an object $J$ in \mA, let us denote by $J_*$ the functor $\V \rightarrow \A$ defined as $J_*X \defeq X * J$ for $X \in \V$.

	\begin{theorem}\label{TheThm}
		Let $\V = (\V, \ox, I, a, \ell, r)$ be a left skew monoidal category and $(\A, *, m, u)$ a strong left \mV-action. Suppose there exists an object $J \in \A$ such that $J_*$ has a left adjoint $J_! \dashv J_*$.
		
		Then there is a\smc{l} structure $(\A, \sk, J, \gamma, \lambda, \rho)$ on \mA defined in the following way. For any $A, B, C \in \A$, we put $A \sk B \defeq J_!A * B$
		\begin{align*}
			\gamma_{A,B,C} &\defeq J_!(J_!A*B)*C \xrightarrow{\tilde{\gamma}_{A,B}*C} (J_!A \ox J_!B)*C \xrightarrow{m^{J_!A, J_!B}_C} J_!A*(J_!B*C) \\
			\lambda_A &\defeq J_!J*A \xrightarrow{J_!u_J^{-1} * A} J_!(I*J)*A \xrightarrow{\epsilon_I * A} I*A \xrightarrow{u_A} A \\ 
			\rho_A &\defeq A \xrightarrow{\eta_A} J_!A*J
		\end{align*}
		where $\eta$ is the unit of $J_! \dashv J_*$ and $\tilde{\gamma}_{A,B}$ is the adjoint transpose of $\big(m^{J_!A, J_!B}_J\big)^{-1} \circ (J_!A * \eta_B)$:
		\begin{prooftree}
			\AxiomC{$J_!A*B \xrightarrow{J_!A * \eta_B} J_!A*(J_!B*J) \xrightarrow{\big(m^{J_!A, J_!B}_J\big)^{-1}} (J_!A \ox J_!B)*J$}
			\RightLabel{$J_! \dashv J_*$}
			\UnaryInfC{$\tilde{\gamma}_{A,B}: J_!(J_!A*B) \rightarrow J_!A \ox J_!B$}
		\end{prooftree}
	\end{theorem}
	\begin{remark}\label{strength}
		Morphisms of the form $J_!(J_!A*B) \rightarrow J_!A \ox J_!B$ are sometimes referred to as a \textit{fusion map}. Observe that the construction of the fusion map $\tilde{\gamma}$ defined in the theorem can be easily generalized, if instead of $J_!A$ we consider arbitrary objects of \mV.
		\begin{prooftree}
			\AxiomC{$X*B \xrightarrow{X * \eta_B} X*(J_!B*J) \xrightarrow{\big(m^{X, J_!B}_J\big)^{-1}} (X \ox J_!B)*J$}
			\RightLabel{$J_! \dashv J_*$}
			\UnaryInfC{$\varsigma_{X,B}: J_!(X*B) \rightarrow X \ox J_!B$}
		\end{prooftree}
		The map $\varsigma$ is called a \textit{strength}. Morphism families of this form play an important role in the definition of functors between \mV-actegories \cite[Definition 2.7]{Szl17}.
		
		A useful observation is that the definition of $\varsigma$ (or $\tilde{\gamma}$) via adjoint transposition entails the two following equations.
		\[\begin{tikzcd}\label{AT*}\tag{AT*}
			{J_!(X*B)} && {X \ox J_!B} \\
			\\
			{J_!(X*(J_!B*J))} && {J_!((X \ox J_!B)*J)}
			\arrow["{\varsigma_{X,B}}", from=1-1, to=1-3]
			\arrow["{J_!(X*\eta_B)}"', from=1-1, to=3-1]
			\arrow["{J_!\inv{m_J^{X,J_!B}}}"', from=3-1, to=3-3]
			\arrow["{\epsilon_{X \ox J_!B}}"', from=3-3, to=1-3]
		\end{tikzcd}\]
		\[\begin{tikzcd}\label{AT}\tag{AT}
			{X*B} && {J_!(X*B)*J} \\
			\\
			{X*(J_!B*J)} && {(X \ox J_!B)*J}
			\arrow["{\eta_{X*B}}", from=1-1, to=1-3]
			\arrow["{X*\eta_B}"', from=1-1, to=3-1]
			\arrow["{\varsigma_{X,B}*J}", from=1-3, to=3-3]
			\arrow["{\inv{m_J^{X,J_!B}}}"', from=3-1, to=3-3]
		\end{tikzcd}\]
		
	\end{remark}
	\begin{proof}
\textbf{(\ref{LSkM1})} The first axiom transcribes to:
		\[\begin{tikzcd}[cramped]
			& {J_!(J_!A*(J_!B*C))*D}              \\
			{J_!(J_!(J_!A*B)*C)*D} && {J_!A*J_!(J_!B*C)*D}  \\
			\\
			{J_!(J_!A*B)*(J_!C*D)} && {J_!A*(J_!B*(J_!C*D))}
			\arrow["{\gamma_{A,J_!B*C, D}}", from=1-2, to=2-3]
			\arrow["{J_!\gamma_{A,B,C} * D}", from=2-1, to=1-2]
			\arrow["{\gamma_{J_!A*B,C,D}}"', from=2-1, to=4-1]
			\arrow["{J_!A * \gamma_{B,C,D}}", from=2-3, to=4-3]
			\arrow["{\gamma_{A,B,J_!C*D}}"', from=4-1, to=4-3]
		\end{tikzcd}\]
		
		If we write out the associator maps $\gamma$ as they are defined, then among others, vertices $J_!((J_!A \ox J_!B)*C)*D$ and $(J_!A \ox (J_!B \ox J_!C))*D$ appear. Notice that these can be connected via the strength $\varsigma_{J_!A \ox J_!B, C}*D$. After expanding the diagram, using the definition of $\gamma$, the diagram looks as follows.
		\[
		\begin{tikzcd}[cramped]
			{{J_!(J_!A*(J_!B*C))*D}}    && {(J_!A \ox J_!(J_!B*C))*D}     && {J_!A*(J_!(J_!B*C)*D)}     \\
			\\
			{J_!((J_!A \ox J_!B)*C)*D}  && {(J_!A \ox (J_!B \ox J_!C))*D} && {J_!A*((J_!B \ox J_!C)*D)} \\
			\\
			{J_!(J_!(J_!A*B)*C)*D}      && {((J_!A \ox J_!B) \ox J_!C)*D} && {J_!A*(J_!B*(J_!C*D))}     \\
			\\
			{(J_!(J_!A*B) \ox J_!C) *D} && {J_!(J_!A*B)*(J_!C*D)}         && {(J_!A \ox J_!B)*(J_!C*D)}
			\arrow["{\tilde{\gamma}_{A,J_!B*C}*D}", from=1-1, to=1-3]
			\arrow["{(J_!A \ox \tilde{\gamma}_{B,C})*D}", from=1-3, to=3-3]
			\arrow["{m^{J_!A,J_!(J_!B*C)}_D}", from=1-3, to=1-5]
			\arrow["{J_!m^{J_!A,J_!B}_C*D}", from=3-1, to=1-1]
			\arrow["{\varsigma_{J_!A \ox J_!B, C} * D}", from=3-1, to=5-3]
			\arrow["{a_{J_!A,J_!B,J_!C}*D}"', from=5-3, to=3-3]
			\arrow["{m^{J_!A,J_!B \ox J_!C}_{D}}", from=3-3, to=3-5]
			\arrow["{J_!A*(\tilde{\gamma}_{B,C}*D)}", from=1-5, to=3-5]
			\arrow["{J_!(\tilde{\gamma}_{A,B}*C)*D}", from=5-1, to=3-1]
			\arrow["{\tilde{\gamma}_{J_!A*B,C}*D}"', from=5-1, to=7-1]
			\arrow["{m^{J_!A \ox J_!B, J_!C}_{D}}", from=5-3, to=7-5]
			\arrow["{J_!A*m^{J_!B,J_!C}_D}", from=3-5, to=5-5]
			\arrow["{(\tilde{\gamma}_{A,B} \ox J_!C) * D}"', from=7-1, to=5-3]
			\arrow["{m^{J_!(J_!A*B),J_!C}_D}"', from=7-1, to=7-3]
			\arrow["{m^{J_!A,J_!B}_{J_!C*D}}"', from=7-5, to=5-5]
			\arrow["{\tilde{\gamma}_{A,B}*(J_!C*D)}"', from=7-3, to=7-5]
		\end{tikzcd}
		\]
		
		Naturality of $m$ ensures commutativity of the upper right square and the square at the bottom. The pentagon on the right is an instance of (\ref{LAct1}). It hence remains to deal with the square and the pentagon on the left.
		
		For the square, it suffices to show that the following commutes.
		\[\begin{tikzcd}
			{J_!(J_!(J_!A*B)*C)}     && {J_!(J_!A*B) \ox J_!C}     \\
			\\
			{J_!((J_!A \ox J_!B)*C)} && {(J_!A \ox J_!B) \ox J_!C}		
			\arrow["{\tilde{\gamma}_{J_!A*B,C}}", from=1-1, to=1-3]
			\arrow["{\tilde{\gamma}_{A,B} \ox J_!C}", from=1-3, to=3-3]
			\arrow["{J_!(\tilde{\gamma}_{A,B}*C)}"', from=1-1, to=3-1]
			\arrow["{\varsigma_{J_!A \ox J_!B, C}}"', from=3-1, to=3-3]
		\end{tikzcd}\]
		Now, we can take the adjoint transpose of this square. This gives us the following diagram, which can be filled in to form two squares commuting because of naturality of $m$ and functoriality of $*$.
		\[\adjustbox{scale=0.95}{\begin{tikzcd}[cramped]
				{J_!(J_!A*B)*C}     && J_!(J_!A*B)*(J_!C*J)     && {(J_!(J_!A*B) \ox J_!C)*J}     \\
				\\
				{(J_!A \ox J_!B)*C} && (J_!A \ox J_!B)*(J_!C*J) && {((J_!A \ox J_!B) \ox J_!C)*J}		
				\arrow["{J_!(J_!A*B) * \eta_C}", from=1-1, to=1-3]
				\arrow["{\inv{m^{J_!(J_!A*B), J_!C}_J}}", from=1-3, to=1-5]
				\arrow["{(\tilde{\gamma}_{A,B} \ox J_!C)*J}", from=1-5, to=3-5]
				\arrow["{\tilde{\gamma}_{A,B}*C}"', from=1-1, to=3-1]
				\arrow["{(J_!A \ox J_!B) * \eta_C}"', from=3-1, to=3-3]
				\arrow["{\inv{m^{J_!A \ox J_!B, J_!C}_J}}"', from=3-3, to=3-5]
				\arrow["{\tilde{\gamma}_{A,B}*(J_!C*J)}", from=1-3, to=3-3]
		\end{tikzcd}}\]
		
		For the remaining pentagon, it is enough to show that the following commutes.
		\[\begin{tikzcd}\label{PentagonWhichWillBeReferencedLater}
			{J_!(J_!A*(J_!B*C))}       && {J_!A \ox J_!(J_!B * C)}   \\
			                           && {J_!A \ox (J_!B \ox J_!C)} \\
			{J_!((J_!A \ox J_!B) * C)} && {(J_!A \ox J_!B) \ox J_!C}
			\arrow["{\tilde{\gamma}_{A,J_!B*C}}", from=1-1, to=1-3]
			\arrow["{J_!m^{J_!A, J_!B}_C}", from=3-1, to=1-1]
			\arrow["{J_!A \ox \tilde{\gamma}_{B,C}}", from=1-3, to=2-3]
			\arrow["{a_{J_!A, J_!B, J_!C}}"', to=2-3, from=3-3]
			\arrow["{\varsigma_{J_!A \ox J_!B, C}}"', from=3-1, to=3-3]
		\end{tikzcd}\]
		After transposing, we get
		\[\begin{tikzcd}[cramped]
			{J_!A*(J_!B*C)}       && {J_!A * (J_!(J_!B*C)*J)}   && {(J_!A \ox J_!(J_!B*C))*J}     \\
			                      &&                            && {(J_!A \ox (J_!B \ox J_!C))*J} \\
			{(J_!A \ox J_!B) * C} && {(J_!A \ox J_!B)*(J_!C*J)} && {((J_!A \ox J_!B) \ox J_!C)*J}
			\arrow["{m^{J_!A, J_!B}_C}", from=3-1, to=1-1]
			\arrow["{J_!A * \eta_{J_!B*C}}", from=1-1, to=1-3]
			\arrow["{\inv{m^{J_!A, J_!(J_!B*C)}_J}}", from=1-3, to=1-5]
			\arrow["{(J_!A \ox \tilde{\gamma}_{B,C})*J}", from=1-5, to=2-5]
			\arrow["{a_{J_!A, J_!B, J_!C}*J}"', from=3-5, to=2-5]
			\arrow["{(J_!A \ox J_!B) * \eta_{C}}"', from=3-1, to=3-3]
			\arrow["{\inv{m^{J_!A \ox J_!B, J_!C}_J}}"', from=3-3, to=3-5]
		\end{tikzcd}\]
		which can be filled in as
		\[\begin{tikzcd}[cramped]
			{J_!A*(J_!B*C)}       && {J_!A * (J_!(J_!B*C)*J)}   && {(J_!A \ox J_!(J_!B*C))*J}     \\
			\\
			&& {J_!A*((J_!B \ox J_!C)*J)} && {(J_!A \ox (J_!B \ox J_!C))*J} \\
			\\
			&& {J_!A*(J_!B*(J_!C*J))}                                       \\
			\\
			{(J_!A \ox J_!B) * C} && {(J_!A \ox J_!B)*(J_!C*J)} && {((J_!A \ox J_!B) \ox J_!C)*J}
			\arrow["{m^{J_!A, J_!B}_C}", from=7-1, to=1-1]
			\arrow["{J_!A * \eta_{J_!B*C}}", from=1-1, to=1-3]
			\arrow["{\inv{m^{J_!A, J_!(J_!B*C)}_J}}", from=1-3, to=1-5]
			\arrow["{(J_!A \ox \tilde{\gamma}_{B,C})*J}", from=1-5, to=3-5]
			\arrow["{a_{J_!A, J_!B, J_!C}*J}"', from=7-5, to=3-5]
			\arrow["{(J_!A \ox J_!B) * \eta_{C}}"', from=7-1, to=7-3]
			\arrow["{\inv{m^{J_!A \ox J_!B, J_!C}_J}}", from=7-3, to=7-5]
			\arrow["{J_!A*(\tilde{\gamma}_{B,C}*J)}", from=1-3, to=3-3]
			\arrow["{\inv{m^{J_!A, J_!B \ox J_!C}_J}}", from=3-3, to=3-5]
			\arrow["{J_!A*(J_!B*\eta_c)}"', from=1-1, to=5-3]
			\arrow["{J_!A*\inv{m^{J_!B,J_!C}_J}}"', from=5-3, to=3-3]
			\arrow["{\inv{m^{J_!A,J_!B}_{J_!C*J}}}", from=5-3, to=7-3]
		\end{tikzcd}\]
		The pentagon is an instance of \ref{LAct1}, the upper right square and the lower left square commute as $m$ is natural.
		Notice that the upper left square can be obtained by applying $J_!A*(-)$ to the equation (\ref{AT}) where $X \defeq J_!B$ and $B \defeq C$. This concludes the proof of the first axiom.

\textbf{(\ref{LSkM2})} becomes
		\[\begin{tikzcd}
			& {J_!A * B}   \\
			{J_!(J_!A*B)*J} && {J_!A*(J_!B*J)} \\
			\arrow["{\eta_{J_!A*B}}"', from=1-2, to=2-1]
			\arrow["{J_!A*\eta_B}", from=1-2, to=2-3]
			\arrow["{\gamma_{A,B,J}}"', from=2-1, to=2-3]
		\end{tikzcd}\]
		which can be written out as
		\[\begin{tikzcd}
			& {J_!A * B}              \\
			{J_!(J_!A*B)*J} && {J_!A*(J_!B*J)} \\
			& {(J_!A \ox J_!B)*J}
			\arrow["{\eta_{J_!A*B}}"', from=1-2, to=2-1]
			\arrow["{J_!A*\eta_B}", from=1-2, to=2-3]
			\arrow["{\tilde{\gamma}_{A,B}*J}"', from=2-1, to=3-2]
			\arrow["{m^{J_!A,J_!B}_J}"', to=2-3, from=3-2]
		\end{tikzcd}\]
		Observe that this is an instance of (\ref{AT}).

\textbf{(\ref{LSkM3})} transcribes to
		\[\begin{tikzcd}
			{J_!(J_!J*A)*B} && {J_!J*(J_!A*B)} \\
			& {J_!A*B}
			\arrow["{\gamma_{J,A,B}}", from=1-1, to=1-3]
			\arrow["{J_!\lambda_A*B}"', from=1-1, to=2-2]
			\arrow["{\lambda_{J_!A*B}}", from=1-3, to=2-2]	
		\end{tikzcd}\]
		After expanding this, we get
		\[\begin{tikzcd}[cramped]
			{J_!(J_!J*A)*B}     && {(J_!J \ox J_!A)*B}     && {J_!J*(J_!A*B)}     \\
			\\
			{J_!(J_!(I*J)*A)*B} && {(J_!(I*J) \ox J_!A)*B} && {J_!(I*J)*(J_!A*B)} \\
			\\
			{J_!(I*A)*B}        && {(I \ox J_!A)*B}        && {I*(J_!A*B)}        \\
			\\
			&& {J_!A*B}  
			\arrow["{\tilde{\gamma}_{J,A}*B}", from=1-1, to=1-3]
			\arrow["{J_!(J_!u_J^{-1}*A)*B}"', from=1-1, to=3-1]
			\arrow["{J_!(\epsilon_I*A)*B}"', from=3-1, to=5-1]
			\arrow["{J_!u_A*B}"', from=5-1, to=7-3]
			\arrow["{(J_!u_J^{-1} \ox J_!A)*B}"', from=1-3, to=3-3]
			\arrow["{(\epsilon_I \ox J_!A)*B}"', from=3-3, to=5-3]
			\arrow["{J_!u_J^{-1}*(J_!A*B)}", from=1-5, to=3-5]
			\arrow["{\epsilon_I*(J_!A*B)}", from=3-5, to=5-5]
			\arrow["{u_{J_!A*B}}", from=5-5, to=7-3]
			\arrow["{m^{J_!J,J_!A}_B}", from=1-3, to=1-5]
			\arrow["{m^{J_!(I*J),J_!A}_B}", from=3-3, to=3-5]
			\arrow["{m^{I,J_!A}_B}", from=5-3, to=5-5]
			\arrow["{\ell_{J_!A}*B}", from=5-3, to=7-3]
		\end{tikzcd}\]
		The triangle is an instance of (\ref{LAct2}) and the squares on the right side commute because $m$ is natural. It suffices to deal with the following diagram.
		\[\begin{tikzcd}
			{J_!(J_!J*A)} && {J_!J \ox J_!A}     \\
			&& {J_!(I*J) \ox J_!A} \\
			&& {I \ox J_!A}        \\
			{J_!A}        && {J_!A}          
			\arrow["{J_!\lambda_A}"', from=1-1, to=4-1]
			\arrow["{\tilde{\gamma}_{J,A}}", from=1-1, to=1-3]
			\arrow["{J_!u_J^{-1} \ox J_!A}", from=1-3, to=2-3]
			\arrow["{\epsilon_I \ox J_!A}", from=2-3, to=3-3]
			\arrow["{\ell_{J_!A}}", from=3-3, to=4-3]
			\arrow[equal, from=4-1, to=4-3]		         
		\end{tikzcd}\]
		Under transposing, it becomes
		\[\begin{tikzcd}
			{J_!J*A} && {J_!J*(J_!A*J)} && {(J_!J \ox J_!A)*J}     \\
			&&                 && {(J_!(I*J) \ox J_!A)*J} \\
			&&                 && {(I \ox J_!A)*J}        \\
			{A}      &&                 && {J_!A*J}
			\arrow["{\lambda_A}"', from=1-1, to=4-1]
			\arrow["{J_!J*\eta_A}", from=1-1, to=1-3]
			\arrow["{\big(m^{J_!J,J_!A}_J\big)^{-1}}", from=1-3, to=1-5]
			\arrow["{(J_!u_J^{-1} \ox J_!A})*J", from=1-5, to=2-5]
			\arrow["{(\epsilon_I \ox J_!A)*J}", from=2-5, to=3-5]
			\arrow["{\ell_{J_!A}*J}", from=3-5, to=4-5]
			\arrow["{\eta_A}"', from=4-1, to=4-5]
		\end{tikzcd}\]
		If we fill in other maps we have, we obtain
		\[\begin{tikzcd}
			{J_!J*A}     && {J_!J*(J_!A*J)}     && {(J_!J \ox J_!A)*J}     \\
			\\
			{J_!(I*J)*A} && {J_!(I*J)*(J_!A*J)} && {(J_!(I*J) \ox J_!A)*J} \\
			\\
			{I*A}        && {I*(J_!A*J)}        && {(I \ox J_!A)*J}        \\
			\\
			{A}          &&                     && {J_!A*J}
			\arrow["{J_!J*\eta_A}", from=1-1, to=1-3]
			\arrow["{J_!u_J^{-1}*A}"', from=1-1, to=3-1]
			\arrow["{\epsilon_I*A}"', from=3-1, to=5-1]
			\arrow["{u_A}"', from=5-1, to=7-1]
			\arrow["{(J_!u_J^{-1} \ox J_!A)*J}", from=1-5, to=3-5]
			\arrow["{(\epsilon_I \ox J_!A)*J}", from=3-5, to=5-5]
			\arrow["{J_!u_J^{-1}*(J_!A*J)}", from=1-3, to=3-3]
			\arrow["{\epsilon_I*(J_!A*J)}", from=3-3, to=5-3]
			\arrow["{\ell_{J_!A}*J}", from=5-5, to=7-5]
			\arrow["{\inv{m^{J_!J,J_!A}_J}}", from=1-3, to=1-5]
			\arrow["{\inv{m^{J_!(I*J),J_!A}_J}}", from=3-3, to=3-5]
			\arrow["{\inv{m^{I,J_!A}_J}}", from=5-3, to=5-5]
			\arrow["{\eta_A}"', from=7-1, to=7-5]
			\arrow["{J_!(I*J)*\eta_A}", from=3-1, to=3-3]
			\arrow["{I*\eta_A}", from=5-1, to=5-3]
			\arrow["{u_{J_!A*J}}"', from=5-3, to=7-5]
		\end{tikzcd}\]
		In this diagram, the triangle is \ref{LAct2} and commutativity of all the squares follows from naturality of the involved morphism families and functoriality of $*$.
		
		\textbf{(\ref{LSkM4})} now looks as
		\[\begin{tikzcd}
			J &&&& J \\
			{J_!J*J} && {J_!(I*J)*J} && {I*J}
			\arrow[equal, from=1-1, to=1-5]
			\arrow["{\eta_J}"', from=1-1, to=2-1]
			\arrow["{J_!u_J^{-1}*J}"', from=2-1, to=2-3]
			\arrow["{\epsilon_I*J}"', from=2-3, to=2-5]
			\arrow["{u_J}"', from=2-5, to=1-5]
		\end{tikzcd}\]
		This diagram can be seen to commute because of naturality of $\eta$ and one of the triangle identities of $J_! \dashv J_*$, as demonstrated in the following.
		\[\begin{tikzcd}
			& J \\
			{J_!J*J} & {I*J} & J \\
			{J_!(I*J)*J} && {I*J}
			\arrow["{\eta_J}"', from=1-2, to=2-1]
			\arrow["{u_J^{-1}}"', from=1-2, to=2-2]
			\arrow[equal, from=1-2, to=2-3]
			\arrow["{J_!u_J^{-1}*J}"', from=2-1, to=3-1]
			\arrow["{\eta_{I*J}}", from=2-2, to=3-1]
			\arrow[equal, from=2-2, to=3-3]
			\arrow["{\epsilon_I*J}"', from=3-1, to=3-3]
			\arrow["{u_J}"', from=3-3, to=2-3]
		\end{tikzcd}\]

\textbf{(\ref{LSkM5})} The fifth axiom transcribes as 
		\[\begin{tikzcd}
			{J_!A*B}        && {J_!A*B}             \\
			{J_!(J_!A*J)*B} && {J_!A*(J_!J*B)}
			\arrow["{J_!\eta_A*B}"', from=1-1, to=2-1]
			\arrow["{\gamma_{A,J,B}}"', from=2-1, to=2-3]
			\arrow["{J_!A*\lambda_B}"', from=2-3, to=1-3]
			\arrow[equal, from=1-1, to=1-3]
		\end{tikzcd}\]
		This diagram can be filled in as follows.
		\[\begin{tikzcd}[cramped]
			{J_!(J_!A*J)*B} && {(J_!A \ox J_!J)*B}     && {J_!A*(J_!J*B)}     \\
			\\
			&& {(J_!A \ox J_!(I*J))*B} && {J_!A*(J_!(I*J)*B)} \\
			\\
			&& {(J_!A \ox I)*B}        && {J_!A*(I*B)}        \\
			\\
			{J_!A*B}        && {J_!A*B}                && {J_!A*B}
			\arrow["{\tilde{\gamma}_{A,J}*B}", from=1-1, to=1-3]
			\arrow["{J_!\eta_A*B}", from=7-1, to=1-1]
			\arrow["{(J_!A \ox J_!u_J^{-1})*B}"', from=1-3, to=3-3]
			\arrow["{(J_!A \ox \epsilon_I)*B}"', from=3-3, to=5-3]
			\arrow["{J_!A*(J_!u_J^{-1}*B)}", from=1-5, to=3-5]
			\arrow["{J_!A*(\epsilon_I*B)}", from=3-5, to=5-5]
			\arrow["{m^{J_!A,J_!J}_B}", from=1-3, to=1-5]
			\arrow["{m^{J_!A,J_!(I*J)}_B}", from=3-3, to=3-5]
			\arrow["{J_!A*u_B}", from=5-5, to=7-5]
			\arrow["{r_{J_!A}*B}", from=7-3, to=5-3]
			\arrow["{m^{J_!A,I}_B}"', to=5-5, from=5-3]
			\arrow[equal, from=7-1, to=7-3]
			\arrow[equal, from=7-3, to=7-5]
		\end{tikzcd}\]
		The lower right square is \ref{LAct3} and the other squares on the right commute because of naturality of $m$. The remaining pentagon
		\[\begin{tikzcd}
			{J_!A}        && {J_!A}              \\
			&& {J_!A \ox I}        \\
			&& {J_!A \ox J_!(I*J)} \\
			{J_!(J_!A*J)} && {J_!A \ox J_!J}  
			\arrow["{J_!\eta_A}"', from=1-1, to=4-1]
			\arrow["{\tilde{\gamma}_{A,J}}"', from=4-1, to=4-3]
			\arrow["{r_{J_!A}}", from=1-3, to=2-3]
			\arrow["{J_!A \ox J_!u_J^{-1}}"', from=4-3, to=3-3]
			\arrow["{J_!A \ox \epsilon_I}"', from=3-3, to=2-3]
			\arrow[equal, from=1-1, to=1-3]
		\end{tikzcd}\]
		becomes the following under transposing
		\[\begin{tikzcd}
			{A}      &&                 && {(J_!A)*J}              \\
			&&                 && {(J_!A \ox I)*J}        \\
			&&                 && {(J_!A \ox J_!(I*J))*J} \\
			{J_!A*J} && {J_!A*(J_!J*J)} && {(J_!A \ox J_!J)*J}
			\arrow["{\eta_A}", from=1-1, to=1-5]
			\arrow["{\eta_A}"', from=1-1, to=4-1]
			\arrow["{r_{J_!A}*J}", from=1-5, to=2-5]
			\arrow["{(J_!A \ox J_!u_J^{-1})*J}"', from=4-5, to=3-5]
			\arrow["{(J_!A \ox \epsilon_I)*J}"', from=3-5, to=2-5]
			\arrow["{J_!A*\eta_J}"', from=4-1, to=4-3]
			\arrow["{\big(m^{J_!A,J_!B}_J\big)^{-1}}"', from=4-3, to=4-5]
		\end{tikzcd}\]
		After filling in additional morphisms, we get
		\[\begin{tikzcd}
			{J_!A*J}         && {J_!A*(J_!J*J)}     && {(J_!A \ox J_!J)*J}     \\
			\\
			&& {J_!A*(J_!(I*J)*J)} && {(J_!A \ox J_!(I*J))*J} \\
			\\
			&& {J_!A*(I*J)}                                   \\                
			{(J_!A \ox I)*J} &&                     && {(J_!A \ox I)*J}
			\arrow["{J_!A*\eta_J}", from=1-1, to=1-3]
			\arrow["{r_{J_!A}*J}"', from=1-1, to=6-1]
			\arrow["{J_!A*u_J^{-1}}"', from=1-1, to=5-3]
			\arrow["{\inv{m^{J_!A,J_!J}_J}}", from=1-3, to=1-5]
			\arrow["{\inv{m^{J_!A,J_!(I*J)}_J}}", from=3-3, to=3-5]
			\arrow["{\inv{m^{J_!A,I}_J}}", from=5-3, to=6-5]
			\arrow["{J_!A*(J_!u_J^{-1}*J)}", from=1-3, to=3-3]
			\arrow["{J_!A*(\epsilon_I*J)}", from=3-3, to=5-3]
			\arrow["{(J_!A \ox J_!u_J^{-1})*J}", from=1-5, to=3-5]
			\arrow["{(J_!A \ox \epsilon_I)*J}", from=3-5, to=6-5]
			\arrow["{m^{J_!A,I}_J}", to=5-3, from=6-1]
			\arrow[equal, from=6-1, to=6-5]
		\end{tikzcd}\]
		Here the triangle on the lower right is \ref{LAct3}, commutativity of the squares on the right follows from naturality of $m$. The upper left square can be recognized to be the axiom (\ref{LSkM4}).
	\end{proof}

	\begin{remark}
		Note that in the proof of the theorem above, we did not make use of the coherence axioms of the skew monoidal category \mV. Furthermore, observe that the proof does not directly depend on the orientations of the constraint morphisms of \mV. If we reverse $a$, $\ell$ or $r$ and adjust the strong action axioms correspondingly, we would still be able to prove the statement of (\ref{TheThm}) -- in the proof, we do not take adjoint transposes of $a$, $\ell$ and $r$. Hence, we could start with a strong action of a right skew monoidal category \mV.\footnote{Note that the theorem could be applied even for a strong \mV-action, where \mV would have a different combination of the orientations of the constraint morphisms $a, \ell, r$. That is, if \mV would be a generalization of a monoidal category with noninvertible constraint morphisms (with any orientations) and five axioms obtained by adjusting MacLane's five axioms for monoidal categories to given orientations of $a, \ell, r$.}
	\end{remark}
	
	\begin{remark}
		In (\ref{TheThm}), we started with a left action and assumed the existence of a left adjoint of the functor $J_*$. We could also consider right actions or the existence of a right adjoint to the action induced functor. In all cases, we would still obtain a skew monoidal structure on the actegory via an analogous construction.
		
		Let us denote, for an object $J$ in \mA, by $J_*$ the functor induced by a left action
		\[J_*: \V \rightarrow \A, \quad X \mapsto X*J\]
		and by $J^*$ the functor induced by a right action
		\[J^*: \V \rightarrow \A, \quad X \mapsto J*X\]
		We will also use the following notation for functors adjoint to those
		\[  J_! \dashv J_*, \quad J_* \dashv J_\#,\quad J^! \dashv J^*,\quad J^* \dashv J^\#  \]
		
		Under this convention, we can summarize the possible variations of (\ref{TheThm}) in the table (\ref{VarThm}).
		
		\begin{table}[htbp]
			\centering
			\begin{tblr}{
					colspec={l|c|c|},
					hline{2,4,6},
					cell{2}{1} = {r=2}{m},
					cell{4}{1} = {r=2}{m}
				}
				& Left adjoint        & Right adjoint       \\
				Left action  & $J_!A*B$            & $J_\#A*B$           \\
				& left skew monoidal  & right skew monoidal \\
				Right action & $A*J^!B$            & $A*J^\#B$           \\
				& right skew monoidal & left skew monoidal  \\
			\end{tblr}
			\caption{Variations of Theorem (\ref{TheThm}). The first row in each cell shows how the skew-tensor $A \sk B$ is defined, the second row tell what kind of skew monoidal structure is obtained this way.}\label{VarThm}
		\end{table}

	\end{remark}
	
	\begin{remark}\label{resSzl}
		A result closely related to (\ref{TheThm}) was proved by Szlachányi in Proposition 6.1. of \cite{Szl17}. He assumes that \mV is a symmetric monoidal closed category with all small limits and colimits and proves that if \mA is a tensored \mV-category, then for every $R$ in \mA, there is a skew monoidal structure on \mA with unit $R$.
		
		Section 2.2. of that paper demonstrates that tensored \mV-categories are equivalent to \mV-actegories \mA such that for any $A$ in \mA the functor $(-)*A: \V \rightarrow \A$ has a right adjoint. Under this correspondence, Szlachányi's induced skew monoidal structure coincides with the structure defined as in our main result (or rather as in one of the variants described in Table (\ref{VarThm})).
		
		Theorem (\ref{TheThm}) can be therefore seen as a generalization of Szlachányi's result, since it allows actegories in which only one object $J$ in \mA gives rise to a functor $(-)*J$ with an adjoint. This generalization is important to capture the example on functor categories (see Example (\ref{exFun})) or examples of structures arising from warping by an opmonoidal monad (see Example (\ref{exSkW})).
	\end{remark}

	\begin{remark}\label{AdjApp}
		In the assumptions of Theorem (\ref{TheThm}), we start by fixing an object $J$ and ask for the induced functor $J_*$ to have an adjoint. In the case that \mV is monoidal, this assumption can be reformulated in such a way that there is no need to fix an object. Instead, we may consider adjunctions $L \dashv R: \V \rightarrow \A$ and require that a certain condition be satisfied. This approach was suggested to the author by Nathanael Arkor.
		
		Let $(\V,\ox,I)$ be a monoidal category and \mA a \mV-actegory. Because \mV is monoidal, it can be seen as a strong \mV-actegory (even monoidal actegory). Suppose we have an adjunction
		\[\begin{tikzcd}
			\A && \V
			\arrow[""{name=0, anchor=center, inner sep=0}, "{L}", bend left, from=1-1, to=1-3]
			\arrow[""{name=1, anchor=center, inner sep=0}, "{R}", bend left, from=1-3, to=1-1]
			\arrow["\dashv"{anchor=center, rotate=-90}, draw=none, from=0, to=1]
		\end{tikzcd}\]
		where $R$ is a strong morphism of \mV-actegories, i.e. it is equipped with an isomorphism
		\[ R(X \ox Y) \xrightarrow{\cong} X * RY \]
		
		Now, if we set $J \defeq RI$, we have for any $X$ in \mV
		\[ RX \cong R(X \ox I) \cong X * RI = X * J = J_*X \]
		
		Hence, we may apply Theorem (\ref{TheThm}) and define a left skew monoidal structure on \mA by setting $$A \sk B \defeq LA*B.$$

		In fact, one can observe that if \mV is a monoidal category, every strong morphism $R: \V \rightarrow \A$ of \mV-actegories is up to isomorphism of form $(-)*J$ for some $J$ in \mA. (This follows from the fact that \mV is the free strong \mV-actegory on $1$.) Therefore, the assumption that there exists an adjunction $L \dashv R: \V \rightarrow \A$ such that $R$ is a strong morphism of \mV-actegories is indeed equivalent to the assumption that there exists an object $J$ in \mA such that $(-)*J$ has a left adjoint.

	\end{remark}

	\section{Monoidality of $J_! \dashv J_*$}\label{MonoidalitySection}
	
	\begin{proposition}\label{Monoidality}
		The adjunction $J_! \dashv J_*$ as in (\ref{TheThm}) is monoidal with respect to $\sk$ and $\ox$.
	\end{proposition}
	\begin{proof}
		Given an adjunction between monoidal categories, the left adjoint is oplax monoidal if and only if the right adjoint is lax monoidal. A similar statement holds for skew monoidal structures, as it is also an instance of a doctrinal adjunction \cite[Theorem 1.2]{Kell74}. Hence, it is enough
		to show that $J_*$ is a lax monoidal functor.
		
		We can define the following structure morphisms
		\begin{align*}
			\phi_{X,Y}&: J_*X \sk J_*Y \xrightarrow{\epsilon_X * J_*Y} X * J_*Y \xrightarrow{\inv{m^{X,Y}_J}} J_*(X \ox Y) \\
			\iota&: J \xrightarrow{u_J^{-1}} J_*I
		\end{align*}
		
		\textbf{Associativity.} We need to show that the following diagram commutes.
		\[\begin{tikzcd}
			{(J_*X \sk J_*Y) \sk J_*Z} && {J_*X \sk (J_*Y \sk J_*Z)} && {J_*X \sk J_*(Y \ox Z)} \\
			\\
			{J_*(X \ox Y) \sk J_*Z} && {J_*((X \ox Y) \ox Z)} && {J_*(X\ox(Y \ox Z))}
			\arrow["{\gamma_{J_*X,J_*Y,J_*Z}}", from=1-1, to=1-3]
			\arrow["{\phi_{X,Y} \sk J_*Z}"', from=1-1, to=3-1]
			\arrow["{J_*X \sk \phi_{Y,Z}}", from=1-3, to=1-5]
			\arrow["{\phi_{X,Y\ox Z}}", from=1-5, to=3-5]
			\arrow["{\phi_{X \ox Y, Z}}"', from=3-1, to=3-3]
			\arrow["{J_*a_{X,Y,Z}}"', from=3-3, to=3-5]
		\end{tikzcd}\]
		After expanding the diagram and filling it in, we get the following.
		\[\begin{tikzcd}
			{J_!(J_!J_*X*J_*Y)*J_*Z} && {J_!(X*J_*Y)*J_*Z} && {J_!J_*(X \ox Y)*J_*Z} \\
			\\
			{(J_!J_*X \ox J_!J_*Y)*J_*Z} && {(X \ox J_!J_*Y)*J_*Z} \\
			\\
			{J_!J_*X*(J_!J_*Y*J_*Z)} && {X*(J_!J_*Y*J_*Z)} && {(X \ox Y)*J_*Z} \\
			\\
			{J_!J_*X*(Y*J_*Z)} && {X*(Y*J_*Z)} && {J_*((X \ox Y) \ox Z)} \\
			\\
			{J_!J_*X*J_*(Y \ox Z)} && {X*J_*(Y \ox Z)} && {J_*(X \ox (Y \ox Z))}
			\arrow["{J_!(\epsilon_X*J_*Y)*J_*Z}", from=1-1, to=1-3]
			\arrow["{\tilde{\gamma}_{J_*X,J_*Y}*J_*Z}"', from=1-1, to=3-1]
			\arrow["{J_!\inv{m^{X,Y}_J}*J_*Z}", from=1-3, to=1-5]
			\arrow["{\varsigma_{X,J_*Y}*J_*Z}"', from=1-3, to=3-3]
			\arrow["{\epsilon_{X \ox Y}*J_*Z}", from=1-5, to=5-5]
			\arrow["{(\epsilon_X \ox J_!J_*Y)*J_*Z}"', from=3-1, to=3-3]
			\arrow["{m^{J_!J_*X,J_!J_*Y}_{J_*Z}}"', from=3-1, to=5-1]
			\arrow["{(X \ox \epsilon_Y)*J_*Z}", from=3-3, to=5-5]
			\arrow["{\epsilon_X*(J_!J_*Y*J_*Z)}"', from=5-1, to=5-3]
			\arrow["{J_!J_*X*(\epsilon_Y*J_*Z)}"', from=5-1, to=7-1]
			\arrow["{m^{X, J_!J_*Y}_{J_*Z}}"', from=3-3, to=5-3]
			\arrow["{X*(\epsilon_Y*J_*Z)}"', from=5-3, to=7-3]
			\arrow["{\inv{m^{X \ox Y,Z}_J}}", from=5-5, to=7-5]
			\arrow["{\epsilon_X*(Y*J_*Z)}"', from=7-1, to=7-3]
			\arrow["{J_!J_*X*\inv{m^{Y,Z}_J}}"', to=9-1, from=7-1]
			\arrow["{\inv{m^{X,Y}_{J_*Z}}}"', from=7-3, to=5-5]
			\arrow["{X*\inv{m^{Y,Z}_J}}"', from=7-3, to=9-3]
			\arrow["{J_*a_{X,Y,Z}}", from=7-5, to=9-5]
			\arrow["{\epsilon_X*J_*(Y \ox Z)}"', from=9-1, to=9-3]
			\arrow["{\inv{m^{X,Y \ox Z}_J}}"', from=9-3, to=9-5]
		\end{tikzcd}\]
		The pentagon is an instance of the axiom \ref{LAct1}. All remaining squares, except for the upper right one obviously commute because of naturality (in the upper left square, we identify $\tilde{\gamma}_{J_*X,J_*Y}$ with $\varsigma_{J_!J_*X, J_*Y}$).
		
		For the remaining part of the diagram, it is enough to show that the following commutes.
		\[\begin{tikzcd}
			{J_!(X*J_*Y)} & {X \ox J_!J_*Y} \\
			{J_!J_*(X\ox Y)} & {X \ox Y}
			\arrow["{\varsigma_{X,J_*Y}}", from=1-1, to=1-2]
			\arrow["{J_!\inv{m^{X,Y}_J}}"', from=1-1, to=2-1]
			\arrow["{X \ox \epsilon_Y}", from=1-2, to=2-2]
			\arrow["{\epsilon_{X \ox Y}}"', from=2-1, to=2-2]
		\end{tikzcd}\]
		After transposing it, we get the following
		\[\begin{tikzcd}
			{X*J_*Y} && {X*(J_!J_*Y*J)} && {(X \ox J_!J_*Y)*J} \\
			\\
			{J_*(X \ox Y)} &&&& {(X \ox Y)*J}
			\arrow["{X*\eta_{J_*Y}}", from=1-1, to=1-3]
			\arrow["{\inv{m^{X,Y}_J}}"', from=1-1, to=3-1]
			\arrow["{\inv{m^{X,J_!J_*Y}_J}}", from=1-3, to=1-5]
			\arrow["{(X \ox \epsilon_Y)*J}", from=1-5, to=3-5]
			\arrow[equal, from=3-1, to=3-5]
		\end{tikzcd}\]
		Now, we can apply the triangle identity of $J_! \dashv J_*$ to obtain the following diagram, which commutes because of  naturality.
		\[\begin{tikzcd}
			{X*J_*Y} && {X*(J_!J_*Y*J)} && {(X \ox J_!J_*Y)*J} \\
			\\
			&& {X*J_*Y} && {(X \ox Y)*J}
			\arrow["{X*\eta_{J_*Y}}", from=1-1, to=1-3]
			\arrow[equal, from=1-1, to=3-3]
			\arrow["{\inv{m^{X,J_!J_*Y}_J}}", from=1-3, to=1-5]
			\arrow["{X*(\epsilon_Y*J)}", from=1-3, to=3-3]
			\arrow["{(X \ox \epsilon_Y)*J}", from=1-5, to=3-5]
			\arrow["{\inv{m^{X,Y}_J}}"', from=3-3, to=3-5]
		\end{tikzcd}\]
		
		\textbf{Unitality.} For left unitality, we need
		\[\begin{tikzcd}
			{J \sk J_*X} && {J_*I \sk J_*X} \\
			{J_*X}       && {J_*(I \ox X)}
			\arrow["{\iota \sk J_*X}", from=1-1, to=1-3]
			\arrow["{\lambda_{J_*X}}"', from=1-1, to=2-1]
			\arrow["{\phi_{I,X}}", from=1-3, to=2-3]
			\arrow["{J_*\ell_X}", from=2-3, to=2-1]
		\end{tikzcd}\]
		This amounts to the following diagram, where the lower triangle is an instance of (\ref{LAct2}).
		\[\begin{tikzcd}
			{J \sk J_*X}   && {J_*I \sk J_*X} \\
			{J_!J_*I*J_*X}                    \\
			{I*J_*X}       && {I*J_*X}        \\
			{J_*X}         && {J_*(I \ox X)}
			\arrow["{J_!u_J^{-1}*J_*X}", from=1-1, to=1-3]
			\arrow["{J_!u_J^{-1}*J_*X}"', from=1-1, to=2-1]
			\arrow["{\epsilon_I*J_*X}"', from=2-1, to=3-1]
			\arrow["{\epsilon_I*J_*X}", from=1-3, to=3-3]
			\arrow["{u_{J_*X}}"', to=4-1, from=3-1]
			\arrow["{\ell_X*J}", from=4-3, to=4-1]
			\arrow["{\inv{m^{I,X}_J}}", from=3-3, to=4-3]
			\arrow[equal, from=3-1, to=3-3]
		\end{tikzcd}\]
		The right unitality condition states that the following commutes.
		\[\begin{tikzcd}
			{J_*X \sk J} && {J_*X \sk J_*I} \\
			{J_*X}       && {J_*(X \ox I)}
			\arrow["{J_*X \sk \iota}", from=1-1, to=1-3]
			\arrow["{\rho_{J_*X}}", from=2-1, to=1-1]
			\arrow["{\phi_{X,I}}", from=1-3, to=2-3]
			\arrow["{J_*r_X}"', from=2-1, to=2-3]
		\end{tikzcd}\]
		We obtain the following, where the triangle is an instance of (\ref{LAct3})
		\[\begin{tikzcd}
			{J_!J_*X * J} && {J_!J_*X * J_*I} \\
			{J_*X}       && {X*J_*I}        \\
			&& {J_*(X \ox I)}
			\arrow["{J_!J_*X * u_J^{-1}}", from=1-1, to=1-3]
			\arrow["{\rho_{J_*X}}", from=2-1, to=1-1]
			\arrow["{\epsilon_X*J_*I}", from=1-3, to=2-3]
			\arrow["{\inv{m^{X,I}_J}}", from=2-3, to=3-3]
			\arrow["{X*u_J^{-1}}"', from=2-1, to=2-3]
			\arrow["{J_*r_X}"', to=3-3, from=2-1]
		\end{tikzcd}\]
		Using the triangle identity of the adjunction $J_! \dashv J_*$, we get the following, which commutes because of naturality
		\[\begin{tikzcd}
			{J_*X} && {J_!J_*X * J} && {J_!J_*X * J_*I} \\
			\\
			&& {J_*X}        && {X*J_*I} 
			\arrow["{\epsilon_X*J}", from=1-3, to=3-3]       
			\arrow["{J_!J_*X * u_J^{-1}}", from=1-3, to=1-5]
			\arrow["{\eta_{J_*X}}", from=1-1, to=1-3]
			\arrow["{\epsilon_X*J_*I}", from=1-5, to=3-5]
			\arrow["{X*u_J^{-1}}"', from=3-3, to=3-5]
			\arrow[equal, from=1-1, to=3-3]
		\end{tikzcd}\]

	\end{proof}
	
	\begin{remark}
		Again, it can be observed that the proof above does not rely on the specific orientations of $a$, $\ell$ and $r$ and thus applies even when \mV is other then left skew monoidal.
	\end{remark}
	
	\begin{remark}\label{MonRem}
		The oplax monoidal structure on the left adjoint $J_!$ is the following
		\begin{align*}
			\widehat{\phi}_{A,B}&: J_!(A \sk B) = J_!(J_!A*B) \xrightarrow{\tilde{\gamma}_{A,B}} J_!A \ox J_!B \\
			\widehat{\iota}&: J_!J \xrightarrow{J_!u_J^{-1}} J_!(I*J) \xrightarrow{\epsilon_I} I
		\end{align*}
		Observe that the structure maps appear in the definitions of the associator and the left unitor in (\ref{TheThm}), i.e. $\gamma_{A,B,C} = \widehat{\phi}_{A,B}*C$ and $\lambda_A$ is $\widehat{\iota} * A$ postcomposed by $u_A$. Namely, we see that the question of when the associator and the left unitor are invertible is closely related to the question of when the left adjoint $J_!$ is strong monoidal. The precise relationship is captured in the following proposition. 
		
		\begin{proposition}
			Let $(\A,\sk,J)$ be a left skew monoidal category defined as in (\ref{TheThm}).
			\begin{enumerate}[(i)]
				\item If the left adjoint $J_!$ is strong monoidal, then the associator and the left unitor are invertible.
				\item If the fusion morphism $\tilde{\gamma}_{A,B}$ for any $A, B$ and $\epsilon_I$ are invertible, $J_!$ is strong monoidal.
			\end{enumerate}
		\end{proposition}
		
	\end{remark}

	\section{Examples}\label{Examples}
	
	\begin{example}
		Any monoidal category $(\V,\ox,I)$ can be regarded as a strong \mV-actegory. If we fix the unit object $I$, the action-induced functor will be $I_* = (-) \ox I \cong \id$. It has an obvious left adjoint $I_! = \id$ and hence we may apply Theorem (\ref{TheThm}). This way, we get a tensor $I_!A \ox B = A \ox B$ and recover the original monoidal structure $(\V,\ox,I)$.
	\end{example}
	
	\begin{example}\label{exFun}
		The category $[\D,\D]$ of endofunctors on a category \mD has a strict monoidal structure given by composition. This is no longer true if we consider a general functor category $[\C,\D]$. However, $[\D,\D]$ acts on $[\C,\D]$ by postcomposition:
		\begin{align*}
			[\D,\D] \x [\C,\D] &\rightarrow [\C,\D] \\
			(X,        F)    &\mapsto X \circ F
		\end{align*}
		
		Fix $J: \C \rightarrow \D$. If we assume \mC, \mD and $J$ are \uv{nice enough}\footnote{E. g. if \mC is small, \mD is cocomplete and $J$ is arbitrary \cite[Theorem 3.7.2]{Bor94}}, then	left Kan extensions of arbitrary functors $F: \C \rightarrow \D$ along $J$ exist, hence we have an adjunction
		\[\begin{tikzcd}
			{[\C,\D]} && {[\D,\D]}
			\arrow[""{name=0, anchor=center, inner sep=0}, "{\Lan_J}", bend left, from=1-1, to=1-3]
			\arrow[""{name=1, anchor=center, inner sep=0}, "{J_*}", bend left, from=1-3, to=1-1]
			\arrow["\dashv"{anchor=center, rotate=-90}, draw=none, from=0, to=1]
		\end{tikzcd}\]
		This gives rise to a left skew monoidal structure on $[\C,\D]$ with tensor $G \sk F = \Lan_JG \circ F$.
		
		This example appeared first in \cite{ACU10}, where the authors study relative monads.
		If left Kan extensions along a functor $J$ exist, and are furthermore pointwise, then $J$-relative monads can be characterized precisely as monoids in this skew monoidal structure.
		
		Proposition 2.3. of \cite{ACU15} shows that monads on \mD restrict to $J$-relative monads. Under certain conditions, one can also go the other way. The authors define the notion of \textit{well behavedness} conditions for $J$ [\textit{ibid.}, Definition 4.1.]. These conditions ensure that the constraint maps of the induced skew monoidal structure are invertible. In Theorem 4.6 it is shown that given a well behaved $J$, $J$-relative monads extend to ordinary monads on \mD. These results are a special instance of (\ref{Monoidality}) and the discussion in (\ref{MonRem}). The result in Proposition 2.3. follows from the fact that $J_*$ is lax monoidal. If $J$ is well behaved, $J_!$ is strong monoidal, so it sends monoids to monoids, which gives Theorem 4.6.
	\end{example}
	
	\begin{example}\label{exSkW}
		Let $T$ be an opmonoidal monad on a monoidal category $(\V,\ox,I)$.	Opmonoidal structures on a monad allow one to lift the monoidal structure of \mV to the category of $T$-algebras $\V^T$ \cite{Moer02}. This allows us to define a monoidal action of $\V^T$ on \mV given first by applying the forgetful functor $U^T$ in the first component and then using the tensor of \mV.	
		\begin{gather*}
			\V^T \times \V \xrightarrow{U^T \times \V} \V \times \V \rightarrow \V \\
			((A,a), B)     \mapsto                   (A,B)       \mapsto A \ox B
		\end{gather*}
		The unit $I$ of \mV gives rise to a functor
		\begin{align*}
			\V^T &\xrightarrow{I_*} \V \\
			(A,a) &\mapsto A \ox I \cong A
		\end{align*}
		In fact $I_* \cong U^T$, hence we have the adjunction
		\[\begin{tikzcd}
			\V && \V^T
			\arrow[""{name=0, anchor=center, inner sep=0}, "{F^T}", bend left, from=1-1, to=1-3]
			\arrow[""{name=1, anchor=center, inner sep=0}, "{I_* \cong U^T}", bend left, from=1-3, to=1-1]
			\arrow["\dashv"{anchor=center, rotate=-90}, draw=none, from=0, to=1]
		\end{tikzcd}\]
		Therefore, in accordance with Theorem (\ref{TheThm}), we obtain a left skew monoidal structure on \mV by defining  $A \sk B \defeq U^TF^TA \ox B$. Since $UF \cong T$, we may write
		\[
		A \sk B = TA \ox B.
		\]
		Note that similar results hold for the dual notion of monoidal comonads.
	\end{example}
	
	\begin{remark}
		The fact that an opmonoidal monad $T$ on a monoidal category $(\V,\ox,I)$ gives rise to a skew monoidal structure with tensor $TX\ox Y$ has been known and used for a long time. It is an instance of a construction of a skew monoidal structure from a \textit{skew warping} \cite{LS12}.
		
		On a left skew monoidal category $(\V, \ox, I)$, a skew warping consists of an endofunctor $T: \V \rightarrow \V$, an object $K \in \V$ and three (natural families of) morphisms:
		\begin{align*}
			v_{X,Y}&: T(TX \ox Y) \rightarrow TX \ox TY \\
			v_0&: TK \rightarrow I                      \\
			k_X&: X \rightarrow TX \ox K 
		\end{align*}
		satisfying five axioms.
		
		The data of a skew warping closely resemble the data, which are used to define the skew monoidal structure in Theorem (\ref{TheThm}). In fact, the theorem always gives rise to a slightly more general notion of a skew warping, which appeared in \cite{LS15}. For a (skew) left action $*: \V \x \A \rightarrow \A$, a \textit{left skew warping riding the action} consists of a functor $T:\A \rightarrow \V$, an object $K \in \A$ and three (natural families of) morphisms 
		\begin{align*}
			v_{A,B}&: T(TA * B) \rightarrow TA \ox TB \\
			v_0&: TK \rightarrow I                      \\
			k_A&: A \rightarrow TA * K 
		\end{align*}
		satisfying five axioms. This notion is in fact a special instance of a skew warping on a skew bicategory in the sense of \cite{LS14}. Under the notation of Theorem (\ref{TheThm}), we get a skew warping riding the given action if we set $T = J_!$, $K = J$, $v_{A,B} = \tilde{\gamma}_{A,B}$, $v_0 = \epsilon_I \circ Ju_J^{-1}$ and $k_A = \eta_A$. The five axioms of a skew warping riding an action will then essentially coincide with the five axioms of a left skew monoidal category.
		
		Conversely, any skew warping riding an action gives rise to a skew monoidal structure on the actegory via a construction analogous to the one of Theorem (\ref{TheThm}) (see \cite{LS15}). This means another way to approach the proof of the statement in (\ref{TheThm}) is to show that the appropriate data constitute a skew warping riding the action. However, to check that the five axioms for skew warping riding an action in the sense of \cite{LS15} are satisfied is essentially the same as showing that the five axioms for a skew monoidal category hold true, so this approach would not offer any simplification.
	\end{remark}
	
	\begin{example}\label{monoid}
		An elementary example can be given on \mSet, if we consider a monoid $(M, +, 0)$. Because $M$ is both a monoid and a comonoid object in \mSet, the functor $M \x (-)$ has a structure of an opmonoidal monad. Hence, we may equip \mSet with a left skew monoidal structure with tensor $A \sw{M} B = M \x A \x B$ and unit $1$. Constraint maps are given element-wise as
		\begin{align*}
			M \x (M \x A \x B) \x C &\xrightarrow{\gamma_{A,B,C}} M \x A \x (M \x B \x C) \\
			(n, (m, a, b), c)       &\mapsto                      (m+n, a, (n, b, c)) \\
			M \times 1 \times B     &\xrightarrow{\lambda_B} B \\
			(m, *, b)               &\mapsto b \\
			A                       &\xrightarrow{\rho_A} M \times A \times 1\\
			a                       &\mapsto (0, a, *)
		\end{align*}
		
		In this example, it is easy to characterize when exactly the associator is invertible. This is precisely when $(m,n) \mapsto (m+n, m)$ is a bijection. If $M$ is a group, it is the case. If $(m,n) \mapsto (m+n, m)$ is invertible, then for any $m \in M$, there has to exist a pair $(m, n)$ such that $(m, n) \mapsto (0, m)$. This means that $m+n = 0$ and hence $(m, n + m) \mapsto (m + n + m, m) = (m, m)$. At the same time $(m, 0) \mapsto (m, m)$. Hence $n + m = 0$. Therefore $\alpha$ is invertible if and only if $M$ is a group. It is easy to see that the unit constraints are not invertible unless $M$ is trivial.
	\end{example}
	
	\begin{example} This example closely follows \cite{BL20}.
		Let $(\V, \ox, I)$ be a symmetric monoidal category and $B$ a bialgebra in \mV. Similarly to the previous example\footnote{Note that the previous example is an instance of the bialgebra one.}, we may define a new tensor on \mV as $X \sw{B} Y = X \ox B \ox Y$. Because of the algebra structure $(B,\mu,\eta)$, the functor $(-) \ox B$ is a monad. The coalgebra structure $(B,\delta,\epsilon)$ ensures it is opmonoidal. From the Theorem (\ref{TheThm}) we see that $(\V,\sw{B},B)$ is a left skew monoidal structure.
		
		To understand what the constraint morphisms do in this example, it is convenient to use string diagrams. Below, the associator and the unitors are shown.
		
		
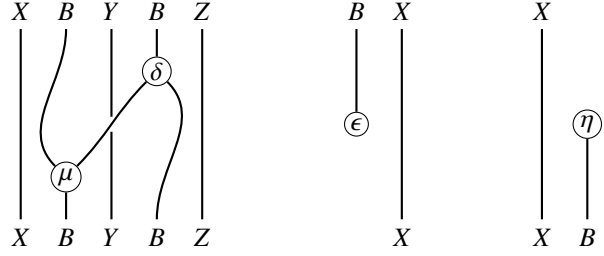
\centerline{
			\begin{tikzpicture}[scale=1]  
				\path (4,0) node {}
				(0,3) node[name=Xu] {$X$}
				(0.6,3) node[name=B1u] {$B$}
				(1.2,3) node[name=Yu] {$Y$}
				(1.8,3) node[name=B2u] {$B$}
				(2.4,3) node[name=Zu] {$Z$}
				(0,0) node[name=Xd] {$X$}
				(0.6,0) node[name=B1d] {$B$}
				(1.2,0) node[name=Yd] {$Y$}
				(1.8,0) node[name=B2d] {$B$}
				(2.4,0) node[name=Zd] {$Z$}
				(0.6,0.8) node[arr, name=m] {$\mu$}
				(1.8,2.2) node[arr, name=d] {$\delta$};
				\draw[braid] (Xu) to (Xd); 
				\draw[braid, name path=Y] (Yu) to (Yd); 
				\draw[braid] (Zu) to (Zd); 
				\draw[braid] (m) to (B1d);
				\draw[braid] (B2u) to (d);
				\draw[braid] (B1u) to [out=270, in=135] (m);
				\draw[braid] (d) to [out=315, in=90] (B2d);
				\path[braid, name path=dm] (d) to [out=225, in=45] (m); 
				\fill[white,name intersections={of=Y and dm}] (intersection-1) circle(0.1); 
				\draw[braid] (d) to [out=225, in=45] (m); 
			\end{tikzpicture}
			\begin{tikzpicture}
				\path (2,0) node {}
				(0.6,3) node[name=Xu] {$X$}
				(0,1.5) node[arr, name=epsilon]  {$\epsilon$}
				(0.6,0) node[name=Xd] {$X$}
				(0,3) node[name=Bu] {$B$} ;
				\draw[braid] (Xu) to (Xd);
				\draw[braid] (epsilon) to (Bu);
			\end{tikzpicture}
			\begin{tikzpicture}
				\path (2,0) node {}
				(0,3) node[name=Xu] {$X$}
				(0.6,1.5) node[arr, name=eta]  {$\eta$}
				(0,0) node[name=Xd] {$X$}
				(0.6,0) node[name=Bd] {$B$} ;
				\draw[braid] (Xu) to (Xd);
				\draw[braid] (eta) to (Bd);
			\end{tikzpicture}
		}
		
		In this example, the associator is invertible if and only if the bialgebra $B$ is Hopf. The unit constraints are invertible only when $B$ is trivial.
		
	\end{example}
	
	\begin{example}
		Let \mV be a monoidal category, which has all copowers of the unit $X \cdot I \defeq \sum_{x \in X}I$ (for $X \in \Set$) and such that the tensor preserves these copowers in the second variable. We have a monoidal adjunction
		\[\begin{tikzcd}
			\Set && \V
			\arrow[""{name=0, anchor=center, inner sep=0}, "{(-)\cdot I}", bend left, from=1-1, to=1-3]
			\arrow[""{name=1, anchor=center, inner sep=0}, "{\V(I, -)}", bend left, from=1-3, to=1-1]
			\arrow["\dashv"{anchor=center, rotate=-90}, draw=none, from=0, to=1]
		\end{tikzcd}\]
		This adjunction then induces a monoidal comonad $K$ on \mV.	The warped monoidal structure $A \sw{K} B \defeq A \ox KB$ is characterized by the fact that
		\begin{prooftree}
			\AxiomC{$A \sw{K} B \rightarrow C$}			
			\UnaryInfC{$\V(I,B) \rightarrow \V(A, C)$}
		\end{prooftree}
	\end{example}
	
	\begin{example}
		Consider $\Vect_\mathbb{K} = (\Vect_\mathbb{K}, \ox, \mathbb{K})$. We have the free-forgetful adjunction
		\[\begin{tikzcd}
			\Set && \Vect_\mathbb{K}
			\arrow[""{name=0, anchor=center, inner sep=0}, "{F = (-)\cdot\mathbb{K}}", bend left, from=1-1, to=1-3]
			\arrow[""{name=1, anchor=center, inner sep=0}, "{U = \Vect_\mathbb{K}(\mathbb{K}, -)}", bend left, from=1-3, to=1-1]
			\arrow["\dashv"{anchor=center, rotate=-90}, draw=none, from=0, to=1]
		\end{tikzcd}\]
		From this, we get a monoidal comonad $FU = \Vect_\mathbb{K}(\mathbb{K}, -)\cdot\mathbb{K}$ and we can define a skew monoidal tensor $A\sw{\mathbb{K}}B \defeq A \ox \Vect_\mathbb{K}(\mathbb{K}, B)\cdot\mathbb{K} = A \ox FUB$.
		
		Because we have $\Vect_\mathbb{K}(A \sw{\mathbb{K}} B, C) = \Vect_\mathbb{K}(A \ox UB \cdot \mathbb{K}, C) \cong \mathsf{Lin}_2(A, UB \cdot \mathbb{K}; C)$, we see that such maps correspond to functions linear in the first variable. 
	\end{example}
	
	\begin{example}
		Let us consider a category \mC with all small coproducts. There is a (left) monoidal action of $(\Set, \times, 1)$ on \mC given by the \mSet-copower \cite[Example 3.2.7.]{CG22}
		\begin{align*}
			\ox: \Set \times \C &\rightarrow \C \\
			\langle X, c \rangle &\mapsto X \ox c \defeq \coprod_{x \in X}c
		\end{align*}
		$X \ox c$ has the universal property
		\begin{equation*}\label{copower}\tag{$\star$}
			\C(X \ox c,d) \cong \Set(X, \C(c,d))
		\end{equation*}
		for any $X \in \Set, c \in \C$. Namely, fixing arbitrary $j \in \C$, $j_\ox = (-)\ox j$ has a right adjoint $j_\# \defeq \C(j, -)$, so applying a variation of Theorem \ref{TheThm} yields a right skew monoidal structure on \mC with tensor
		\[c \sk_\ox d \defeq j_\# c*d = \coprod_{f: j \rightarrow c}d \]

		We could also use the universal property (\ref{copower}) as a defining property of $X \ox c$. In this case we can generalize \mSet to an arbitrary closed monoidal category \mV and talk about \mV-copowered categories.
	\end{example}
	
	\begin{example}
		If we dually consider \mC to be a category with all small products, we can define a (left) monoidal action of $\Set^\mathsf{op}$ on \mC as follows \cite[Example 3.2.8.]{CG22}.
		\begin{align*}
			\pitchfork: \Set^\mathsf{op} \times \C &\rightarrow \C \\
			\langle X, c \rangle &\mapsto X \pitchfork c \defeq \prod_{x \in X}c
		\end{align*}
		$X \pitchfork c$ has the universal property
		\begin{equation*}\label{power}\tag{$\dagger$}
			\Set^\mathsf{op}(\C(d,c),X) \cong \C(d, X \pitchfork c)
		\end{equation*}
		Hence, any $j_\pitchfork = (-)\pitchfork j$ has a left adjoint $j_! \defeq \C(-,j)$, so we can define a left skew monoidal structure with tensor
		\[c \sk_\pitchfork d \defeq  j_!c \pitchfork d = \prod_{g: c \rightarrow j}d\]
	\end{example}
	
	\begin{example}\cite[Example 3.2.9.]{CG22}
		Suppose now that \mM and \mC are small categories. If \mM is a monoidal category, then we can define a monoidal structure on $[\M,\Set]$ by Day convolution.
		Then, we can also extend any monoidal action
		\begin{align*}
			\M \times \C &\rightarrow \C \\
			(m, c)       &\mapsto m \bullet c
		\end{align*}
		to a monoidal action of $[\M,\Set]$ on $[\C,\Set]$ called the \textit{Day convolaction} given as
		\begin{equation*}
			M * C \defeq \int^{m\in\M,c\in\C}\C(m \bullet c, -) \times Mm \times Cc
		\end{equation*}
		with unit $\M(j,-)$.
		
		For $J \in [\C,\Set]$ we have
		\[J_*M = \int^{m\in\M,c\in\C}\C(m \bullet c, -) \times Mm \times Jc\]
		
		If we denote by $\ox_d$ the Day convolution tensor product on $[\M,\Set]$, then the functor $(-)\ox_d  G$ for a fixed $G\in[\M,\Set]$ has a right adjoint \cite[Remark 6.2.4.]{Lor23} of form
		\[F \mapsto \int_{m\in\M}\big[Gm, F(m \ox -)\big] \]
		
		We can adjust this result to our action setting. For a fixed $J \in [\C,\Set]$, the functor $J_* \defeq (-)*J$ has a right adjoint\footnote{Checking that this functor is indeed a right adjoint amounts to adjusting the proof of closedness of the convolution tensor, which can be found for instance in \cite{Lor23}.} $J_\#$ defined as
		\[J_\#C \defeq \int_{c}[Jc, C(- \bullet c)].\]
		
		Therefore, we can apply (\ref{TheThm}) to equip $[\C, \Set]$ with a skew monoidal tensor
		\[
		C \sk D \defeq J_\#C*D = \int^{m\in\M,c\in\C} \C(m \bullet c, -) \x Dc \x \int_d\Set(Jd,C(m \bullet c))
		\]
		
	\end{example}

	\section{Braidings on induced structures}\label{Braiding}
	
	The usual notion of a \textit{braiding} on a monoidal category $(\V, \ox, I)$ involves isomorphisms indexed by pairs of objects: $c_{X,Y}: X \ox Y \rightarrow Y \ox X$. Generalizing the notion to the setting of skew monoidal categories is not entirely straightforward. In \cite{BL20}, Bourke and Lack suggest that on a skew monoidal category, braidings should involve isomorphisms indexed by triples of objects, where one object acts as a \uv{pivot}: $s^P_{A,B}: (P \sk A) \sk B \rightarrow (P \sk B) \sk A$. In the case of a monoidal category, this notion of a braiding coincides with the usual one [\textit{ibid}., prop. 2.6].
	
	However, there is a priori no reason why one should consider only braidings where the pivot is on the left side. One could also define a braiding on a skew monoidal category as $s^P_{A,B}: A \sk (B \sk P) \rightarrow B \sk (A \sk P)$. Note that if we take a left skew monoidal category $(\A,\sk,I)$ braided in the sense of \cite{BL20}, we will have a left skew monoidal category $(\A^\mathsf{op}, \sk^\mathsf{rev}, I)$ with this alternative notion of a braiding. Both of these variants can also be considered on right skew monoidal categories, which can be easily justified by looking at these structures in an opposite category.
	
	Therefore, we will introduce two variants of the definition of a braiding -- a \textit{right braiding} (which corresponds to the braiding defined in \cite{BL20}) and a \textit{left braiding}.

	\begin{definition}
		A \textit{right braiding} on a left skew monoidal category $(\A, \sk, J, \gamma, \lambda, \rho)$ consists of a natural isomorphism
		\begin{equation*}
			s^P_{A,B}: (P \sk A) \sk B \rightarrow (P \sk B) \sk A
		\end{equation*}
		satisfying the following four axioms:
		\[\begin{tikzcd}\label{SkBr1}\tag{RSkBr1}
			{((P \sk A) \sk B) \sk C} && {((P \sk A) \sk C) \sk B} && {((P \sk C) \sk A) \sk B} \\
			\\
			{((P \sk B) \sk A) \sk C} && {((P \sk B) \sk C) \sk A} && {((P \sk C) \sk B) \sk A}
			\arrow["{s^{P \sk A}_{B,C}}", from=1-1, to=1-3]
			\arrow["{s^{P}_{A,C} \sk B}", from=1-3, to=1-5]
			\arrow["{s^{P \sk C}_{A,B}}", from=1-5, to=3-5]
			\arrow["{s^{P}_{A,B} \sk C}"', from=1-1, to=3-1]
			\arrow["{s^{P \sk B}_{A,C}}"', from=3-1, to=3-3]
			\arrow["{s^{P}_{B,C} \sk A}"', from=3-3, to=3-5]
		\end{tikzcd}\]
		\[\begin{tikzcd} \label{SkBr2} \tag{RSkBr2}
			             & {((P \sk B) \sk A) \sk C}               \\
			{((P \sk A) \sk B) \sk C} && {((P \sk B) \sk C) \sk A} \\
			\\
			{(P \sk A) \sk (B \sk C)} && {(P \sk (B \sk C)) \sk A}
			\arrow["{s^{P \sk B}_{A,C}}", from=1-2, to=2-3]
			\arrow["{s^P_{A,B} \sk C}", from=2-1, to=1-2]
			\arrow["{\gamma_{P \sk A,B,C}}"', from=2-1, to=4-1]
			\arrow["{\gamma_{P,B,C} \sk A}", from=2-3, to=4-3]
			\arrow["{s^{P}_{A, B \sk C}}"', from=4-1, to=4-3]
		\end{tikzcd}\]
		\[\begin{tikzcd} \label{SkBr3} \tag{RSkBr3}
			             & {((P \sk A) \sk C) \sk B}               \\
			{((P \sk A) \sk B) \sk C} && {((P \sk C) \sk A) \sk B} \\
			\\
			{(P \sk (A \sk B)) \sk C} && {(P \sk C) \sk (A \sk B)}
			\arrow["{s^{P}_{A,C} \sk B}", from=1-2, to=2-3]
			\arrow["{s^{P \sk A}_{B,C}}", from=2-1, to=1-2]
			\arrow["{\gamma_{P,A,B} \sk C}"', from=2-1, to=4-1]
			\arrow["{\gamma_{P \sk C,A,B}}", from=2-3, to=4-3]
			\arrow["{s^{P}_{A \sk B, C}}"', from=4-1, to=4-3]
		\end{tikzcd}\]
		\[\begin{tikzcd}\label{SkBr4}\tag{RSkBr4}
			{((P \sk A) \sk B) \sk C} && {(P \sk (A \sk B)) \sk C} && {P \sk ((A \sk B) \sk C)} \\
			\\
			{((P \sk A) \sk C) \sk B} && {(P \sk (A \sk C)) \sk B} && {P \sk ((A \sk C) \sk B)}
			\arrow["{\gamma_{P,A,B} \sk C}", from=1-1, to=1-3]
			\arrow["{\gamma_{P,A \sk B, C}}", from=1-3, to=1-5]
			\arrow["{P \sk s^{A}_{B,C}}", from=1-5, to=3-5]
			\arrow["{s^{P \sk A}_{B,C}}"', from=1-1, to=3-1]
			\arrow["{\gamma_{P,A,C} \sk B}"', from=3-1, to=3-3]
			\arrow["{\gamma_{P,A \sk C, B}}"', from=3-3, to=3-5]
		\end{tikzcd}\]
		
		We say that $s$ is a \textit{right symmetry} if
		\[\begin{tikzcd}
			{(P \sk A) \sk B} && {(P \sk A) \sk B} \\
			& {(P \sk B) \sk A}
			\arrow[equal, from=1-1, to=1-3]
			\arrow["{s^P_{A,B}}"', from=1-1, to=2-2]
			\arrow["{s^P_{B,A}}"', from=2-2, to=1-3]
		\end{tikzcd}\]
		
		Analogously, we define a \textit{left braiding} on a left skew monoidal category as a family of natural isomorphisms
		\begin{equation*}
			s^P_{A,B}: A \sk (B \sk P) \rightarrow B \sk (A \sk P)
		\end{equation*}
		satisfying four axioms corresponding to (\ref{SkBr1}) -- (\ref{SkBr4}). Similarly, we say that such $s$ is a \textit{left symmetry}, if it is self-inverse in the sense that $s^P_{B,A} \circ s^P_{A,B} = \id_{A \sk (B \sk P)}$.
		
	\end{definition}

	We can show that for skew monoidal structures arising from actions as in (\ref{TheThm}), usual braidings on the acting monoidal category give rise to braidings on the induced skew monoidal structure. In the following, we will consider a structure arising on a right actegory from a right adjoint, as this setting corresponds to the setting of \cite{BL20}.
	
	\begin{proposition}
		Let $(\V, \ox, I, a, \ell, r, c)$ be a braided monoidal category and \mA a strong right \mV-actegory. Suppose that $J^*$ has a right adjoint $J^\#$ and let $(\A, \sk, J, \gamma, \lambda, \rho)$ be the left skew monoidal category constructed from a right action and a right adjoint as in (\ref{VarThm}). Then the morphism family defined as the following composite
		\[\begin{tikzcd}
			{(P \sk A) \sk B = (P*J^\#A)*J^\#B} && {P * (J^\#A \ox J^\#B)} \\
			{(P \sk B) \sk A = (P*J^\#B)*J^\#A} && {P * (J^\#B \ox J^\#A)}
			\arrow["{m^{J^\#A, J^\#B}_P}", from=1-1, to=1-3]
			\arrow["{P*c_{J^\#A,J^\#B}}", from=1-3, to=2-3]
			\arrow["{\inv{m^{J^\#B, J^\#A}_P}}", from=2-3, to=2-1]
			\arrow["{s^P_{A,B}}"', from=1-1, to=2-1]
		\end{tikzcd}\]
		is a right braiding on \mA.
		
		Furthermore, if $c$ is a symmetry, then $s$ is a right symmetry.
	\end{proposition}
	\begin{proof}
		The first axiom looks as
		\[\begin{tikzcdscale}
			{((P*J^\#A)*J^\#B)*J^\#C} && {((P*J^\#A)*J^\#C)*J^\#B} && {((P*J^\#C)*J^\#A)*J^\#B} \\
			\\
			{((P*J^\#B)*J^\#A)*J^\#C} && {((P*J^\#B)*J^\#C)*J^\#A} && {((P*J^\#C)*J^\#B)*J^\#A}
			\arrow["{s^{P \sk A}_{B,C}}", from=1-1, to=1-3]
			\arrow["{s^{P}_{A,C} \sk B}", from=1-3, to=1-5]
			\arrow["{s^{P \sk C}_{A,B}}", from=1-5, to=3-5]
			\arrow["{s^{P}_{A,B} \sk C}"', from=1-1, to=3-1]
			\arrow["{s^{P \sk B}_{A,C}}"', from=3-1, to=3-3]
			\arrow["{s^{P}_{B,C} \sk A}"', from=3-3, to=3-5]
		\end{tikzcdscale}\]
		After expanding using definitions of involved morphisms, we get the diagram \ref{pfSkBr1}. Here, all inner pentagons are instances of (\ref{LAct1}), all inner squares commute because of naturality of the involved morphisms and the hexagons are instances of one of the hexagon identities for $c$.
		
		(\ref{SkBr2}) translates to
		\[\begin{tikzcd}
			             & {((P*J^\#B)*J^\#A)*J^\#C}               \\
			{((P*J^\#A)*J^\#B)*J^\#C} && {((P*J^\#B)*J^\#C)*J^\#A} \\
			\\
			{(P*J^\#A)*J^\#(B*J^\#C)} && {(P*J^\#(B*J^\#C))*J^\#A}
			\arrow["{s^{P \sk B}_{A,C}}", from=1-2, to=2-3]
			\arrow["{s^P_{A,B} \sk C}", from=2-1, to=1-2]
			\arrow["{\gamma_{P \sk A,B,C}}"', from=2-1, to=4-1]
			\arrow["{\gamma_{P,B,C} \sk A}", from=2-3, to=4-3]
			\arrow["{s^{P}_{A,B \sk C}}"', from=4-1, to=4-3]
		\end{tikzcd}\]
		which can be expanded and filled in as (\ref{pfSkBr2}). Here, the pentagons are again (\ref{LAct1}), the hexagon is an instance of an axiom for $c$ and the squares commute because of naturality.
		
		The third axiom (\ref{SkBr3}) now follows from what has been shown already. Since $c$ is a braiding on \mV, we also have the inverse braiding $\tilde{c}_{X,Y} = c_{Y,X}^{-1}$. If we replace $c$ with $\tilde{c}$ in the definition in the statement of the proposition, we obtain a morphism family $\tilde{s}$ for which (\ref{SkBr2}) holds as well. Observe that (\ref{SkBr2}) for $\tilde{s}$ is equivalent to (\ref{SkBr3}) for $s$.
		
		The fourth axiom (\ref{SkBr4}) is
		\[\begin{tikzcdscale}
			{((P*J^\#A)*J^\#B)*J^\#C} && {(P*J^\#(A*J^\#B))*J^\#C} && {P*J^\#((A*J^\#B)*J^\#C)} \\
			\\
			{((P*J^\#A)*J^\#C)*J^\#B} && {(P*J^\#(A*J^\#C))*J^\#B} && {P*J^\#((A*J^\#C)*J^\#B)}
			\arrow["{\gamma_{P,A,B} \sk C}", from=1-1, to=1-3]
			\arrow["{\gamma_{P,A \sk B, C}}", from=1-3, to=1-5]
			\arrow["{P \sk s^{A}_{B,C}}", from=1-5, to=3-5]
			\arrow["{s^{P \sk A}_{B,C}}"', from=1-1, to=3-1]
			\arrow["{\gamma_{P,A,C} \sk B}"', from=3-1, to=3-3]
			\arrow["{\gamma_{P,A \sk C, B}}"', from=3-3, to=3-5]
		\end{tikzcdscale}
		\]
		which can be written out as (\ref{pfSkBr4}). Here, the squares commute because of naturality and the two pentagons on the left are instances of (\ref{LAct1}). It remains to deal with the two pentagons on the right. It suffices to show that the following commutes.
		\[\begin{tikzcd}
			{J^\#(A * J^\#B) \ox J^\#C}    && {J^\#((A * J^\#B) * J^\#C)}   \\
			{(J^\#A \ox J^\#B) \ox J^\#C}                                   \\
			{J^\#A \ox (J^\#B \ox J^\#C)}  && {J^\#(A * (J^\#B \ox J^\#C))}
			\arrow["{a_{J^\#A, J^\#B, J^\#C}}"', from=2-1, to=3-1]
			\arrow["{\tilde{\gamma}_{A,B} \ox J^\#C}", from=2-1, to=1-1]
			\arrow["{\tilde{\gamma}_{A*J^\#B, C}}", from=1-1, to=1-3]
			\arrow["{\varsigma_{A,J^\#B \ox J^\#C}}"', from=3-1, to=3-3]
			\arrow["{J^\#m^{J^\#B,J^\#C}_A}", from=1-3, to=3-3]
		\end{tikzcd}\]
		where $\varsigma_{A,Y}$ is the right action and right adjoint version of the strength introduced in (\ref{strength}) (namely $\tilde{\gamma}_{A,B} = \varsigma_{A, J^\# B} $).
		Observe that this diagram is the right action and right adjoint version of the pentagon which appears on page \pageref{PentagonWhichWillBeReferencedLater} and whose commutativity is shown as a part of the proof of theorem (\ref{TheThm}).

		To complete the proof of the statement, it remains to show that if $c$ is a symmetry, the for every $P,A,B \in \A$, $s^P_{B,A} \circ s^P_{A,B} = \id$. This immediately follows, as shown by the diagram below.
		\[\begin{tikzcd}
			{(P*J^\#A)*J^\#B} && {P * (J^\#A \ox J^\#B)} \\
			&& {P * (J^\#B \ox J^\#A)} \\
			{(P*J^\#B)*J^\#A}                            \\
			&& {P * (J^\#B \ox J^\#A)} \\
			{(P*J^\#A)*J^\#B} && {P * (J^\#A \ox J^\#B)}
			\arrow["{m^{J^\#A, J^\#B}_P}", from=1-1, to=1-3]
			\arrow["{P*c_{J^\#A,J^\#B}}", from=1-3, to=2-3]
			\arrow["{P*c_{J^\#B,J^\#A}}", from=4-3, to=5-3]
			\arrow["{\inv{m^{J^\#B, J^\#A}_P}}"', from=2-3, to=3-1]
			\arrow["{m^{J^\#B, J^\#A}_P}"', from=3-1, to=4-3]
			\arrow[equal, from=2-3, to=4-3]
			\arrow["{s^P_{A,B}}"', from=1-1, to=3-1]
			\arrow["{s^P_{B,A}}"', from=3-1, to=5-1]
			\arrow["{\inv{m^{J^\#A, J^\#B}_P}}", from=5-3, to=5-1]
			\arrow["{\id}", bend left=45, shift left=15, from=1-3, to=5-3]
		\end{tikzcd}\]
		
	\end{proof}	
	
	\newgeometry{top=5mm, bottom=5mm}
	
	\begin{landscape}
		\begin{equation}\label{pfSkBr1}
			\begin{tikzcdscale}[cramped, font = \large]
				{((P*J^\#A)*J^\#B)*J^\#C} &&&& {(P*J^\#A)*(J^\#B \ox J^\#C)} && {(P*J^\#A)*(J^\#C \ox J^\#B)} &&&& {((P*J^\#A)*J^\#C)*J^\#B} \\
				\\
				{(P*(J^\#A \ox J^\#B))*J^\#C} &&&& {P*(J^\#A \ox (J^\#B \ox J^\#C))} && {P*(J^\#A \ox (J^\#C \ox J^\#B))} &&&& {(P*(J^\#A \ox J^\#C))*J^\#B} \\
				&& {P*((J^\#A \ox J^\#B) \ox J^\#C)} &&&&&& {P*((J^\#A \ox J^\#C) \ox J^\#B)} \\
				{(P*(J^\#B \ox J^\#A))*J^\#C} &&&&&&&&&& {(P*(J^\#C \ox J^\#A))*J^\#B} \\
				&& {P*((J^\#B \ox J^\#A) \ox J^\#C)} &&&&&& {P*((J^\#C \ox J^\#A) \ox J^\#B)} \\
				{((P*J^\#B)*J^\#A)*J^\#C} &&&&&&&&&& {((P*J^\#C)*J^\#A)*J^\#B} \\
				&& {P*(J^\#B \ox (J^\#A \ox J^\#C))} &&&&&& {P*(J^\#C \ox (J^\#A \ox J^\#B))} \\
				{(P*J^\#B)*(J^\#A \ox J^\#C)} &&&&&&&&&& {(P*J^\#C)*(J^\#A \ox J^\#B)} \\
				&& {P*(J^\#B \ox (J^\#C \ox J^\#A))} &&&&&& {P*(J^\#C \ox (J^\#B \ox J^\#A))} \\
				{(P*J^\#B)*(J^\#C \ox J^\#A)} &&&& {P*((J^\#B \ox J^\#C) \ox J^\#A)} && {P*((J^\#C \ox J^\#B) \ox J^\#A)} &&&& {(P*J^\#C)*(J^\#B \ox J^\#A)} \\
				\\
				{((P*J^\#B)*J^\#C)*J^\#A} &&&& {(P*(J^\#B \ox J^\#C))*J^\#A} && {(P*(J^\#C \ox J^\#B))*J^\#A} &&&& {((P*J^\#C)*J^\#B)*J^\#A}
				\arrow["{m^{J^\#B,J^\#C}_{P*J^\#A}}", from=1-1, to=1-5]
				\arrow["{m^{J^\#A,J^\#B}_{P}*J^\#C}"', from=1-1, to=3-1]
				\arrow["{(P*J^\#A)*c_{J^\#B, J^\#C}}", bend left, from=1-5, to=1-7]
				\arrow["{m^{J^\#A, J^\#B \ox J^\#C}_P}"', from=1-5, to=3-5]
				\arrow["{\inv{m^{J^\#C, J^\#B}_{P*J^\#A}}}", from=1-7, to=1-11]
				\arrow["{m^{J^\#A, J^\#C \ox J^\#B}_P}", from=1-7, to=3-7]
				\arrow["{m^{J^\#A,J^\#C}_P * J^\#B}", from=1-11, to=3-11]
				\arrow["{m^{J^\#A \ox J^\#B, J^\#C}_P}"', from=3-1, to=4-3]
				\arrow["{(P * c_{J^\#A,J^\#B})*J^\#C}"', from=3-1, to=5-1]
				\arrow["{P*(J^\#A \ox c_{J^\#B, J^\#C})}"', bend right, from=3-5, to=3-7]
				\arrow["{P*c_{J^\#A,J^\#B \ox J^\#C}}"{description}, from=3-5, to=11-5]
				\arrow["{P*c_{J^\#A,J^\#C \ox J^\#B}}"{description}, from=3-7, to=11-7]
				\arrow["{\inv{m^{J^\#A \ox J^\#C, J^\#B}_P}}"', to=3-11, from=4-9]
				\arrow["{(P*c_{J^\#A,J^\#C}) * J^\#B}", from=3-11, to=5-11]
				\arrow["{P*a_{J^\#A, J^\#B, J^\#C}}"', from=4-3, to=3-5]
				\arrow["{P * (c_{J^\#A,J^\#B} \ox J^\#C)}", from=4-3, to=6-3]
				\arrow["{P*a_{J^\#A, J^\#C, J^\#B}^{-1}}"', to=4-9, from=3-7]
				\arrow["{P * (c_{J^\#A,J^\#C} \ox J^\#B)}"', from=4-9, to=6-9]
				\arrow["{m^{J^\#B \ox J^\#A, J^\#C}_P}"', from=5-1, to=6-3]
				\arrow["{\inv{m^{J^\#B,J^\#A}_P}*J^\#C}"', from=5-1, to=7-1]
				\arrow["{m^{J^\#C \ox J^\#A, J^\#B}_P}", from=5-11, to=6-9]
				\arrow["{\inv{m^{J^\#C,J^\#A}_P}*J^\#B}", from=5-11, to=7-11]
				\arrow["{P*a_{J^\#B, J^\#A, J^\#C}}", from=6-3, to=8-3]
				\arrow["{P*a_{J^\#C, J^\#A, J^\#B}}"', from=6-9, to=8-9]
				\arrow["{m^{J^\#A,J^\#C}_{P*J^\#B}}"', from=7-1, to=9-1]
				\arrow["{m^{J^\#A,J^\#B}_{P*J^\#C}}", from=7-11, to=9-11]
				\arrow["{P*(J^\#B \ox c_{J^\#A,J^\#C})}", from=8-3, to=10-3]
				\arrow["{P*(J^\#C \ox c_{J^\#A,J^\#B})}"', from=8-9, to=10-9]
				\arrow["{m^{J^\#B, J^\#A \ox J^\#C}_P}"', from=9-1, to=8-3]
				\arrow["{(P*J^\#B)*c_{J^\#A,J^\#C}}"', from=9-1, to=11-1]
				\arrow["{m^{J^\#C, J^\#A \ox J^\#B}_P}", from=9-11, to=8-9]
				\arrow["{(P*J^\#C)*c_{J^\#A,J^\#B}}", from=9-11, to=11-11]
				\arrow["{m^{J^\#B, J^\#C \ox J^\#A}_P}"', from=11-1, to=10-3]
				\arrow["{\inv{m^{J^\#C,J^\#A}_{P*J^\#B}}}"', from=11-1, to=13-1]
				\arrow["{P*a_{J^\#B, J^\#C, J^\#A}^{-1}}"', to=11-5, from=10-3]
				\arrow["{P*(c_{J^\#B, J^\#C} \ox J^\#A)}", bend left, from=11-5, to=11-7]
				\arrow["{P*a_{J^\#C, J^\#B, J^\#A}}"', from=11-7, to=10-9]
				\arrow["{\inv{m^{J^\#C,J^\#B \ox J^\#A}_{P}}}"', to=11-11, from=10-9]
				\arrow["{\inv{m^{J^\#B,J^\#A}_{P*J^\#C}}}", from=11-11, to=13-11]
				\arrow["{m^{J^\#B,J^\#C}_P*J^\#A}"', from=13-1, to=13-5]
				\arrow["{m^{J^\#B \ox J^\#C, J^\#A}_P}", from=13-5, to=11-5]
				\arrow["{(P*c_{J^\#B, J^\#C}) * J^\#A}"', bend right, from=13-5, to=13-7]
				\arrow["{m^{J^\#C \ox J^\#B, J^\#A}_P}"', from=13-7, to=11-7]
				\arrow["{\inv{m^{J^\#C,J^\#B}_P}*J^\#A}"', from=13-7, to=13-11]
			\end{tikzcdscale}
		\end{equation}
		
		\begin{equation}\begin{tikzcdscale}[cramped]\label{pfSkBr2}
				&&& {((P*J^\#B)*J^\#A)*J^\#C}\\
				\\
				{(P*(J^\#B \ox J^\#A))*J^\#C} &&&&&& {(P*J^\#B)*(J^\#A \ox J^\#C)} \\
				&& {P*((J^\#B \ox J^\#A) \ox J^\#C)} && {P*(J^\#B \ox (J^\#A \ox J^\#C))}\\
				{(P*(J^\#A \ox J^\#B))*J^\#C} &&&&&& {(P*J^\#B)*(J^\#C \ox J^\#A)} \\
				&& {P*((J^\#A \ox J^\#B) \ox J^\#C)} && {P*(J^\#B \ox (J^\#C \ox J^\#A))}\\
				{((P*J^\#A)*J^\#B)*J^\#C} &&&&&& {((P*J^\#B)*J^\#C)*J^\#A} \\
				&& {P*(J^\#A \ox (J^\#B \ox J^\#C))} && {P*((J^\#B \ox J^\#C) \ox J^\#A)}\\
				{(P*J^\#A)*(J^\#B \ox J^\#C)} &&&&&& {(P*(J^\#B \ox J^\#C))*J^\#A} \\
				&& {P*(J^\#A \ox J^\#(B * J^\#C))} && {P*(J^\#(B*J^\#C) \ox J^\#A)}\\
				{(P*J^\#A)*J^\#(B*J^\#C)} &&&&&& {(P*J^\#(B*J^\#C))*J^\#A} \\
				\arrow["{\inv{m^{J^\#B,J^\#A}_P}*J^\#C}", from=3-1, to=1-4]
				\arrow["{m^{J^\#A,J^\#C}_{P*J^\#B}}", from=1-4, to=3-7]
				\arrow["{m^{J^\#B \ox J^\#A, J^\#C}_P}", from=3-1, to=4-3]
				\arrow["{P*a_{J^\#B,J^\#A,J^\#C}}", from=4-3, to=4-5]
				\arrow["{(P*c_{J^\#A,J^\#B})*J^\#C}", from=5-1, to=3-1]
				\arrow["{m^{J^\#A \ox J^\#B, J^\#C}_P}"', from=5-1, to=6-3]
				\arrow["{P*(c_{J^\#A,J^\#B} \ox J^\#C)}", from=6-3, to=4-3]
				\arrow["{\inv{m^{J^\#B, J^\#A \ox J^\#C}_P}}", to=3-7, from=4-5]
				\arrow["{(P*J^\#B)*c_{J^\#A,J^\#C}}", from=3-7, to=5-7]
				\arrow["{\inv{m^{J^\#B, J^\#C \ox J^\#A}_P}}"', to=5-7, from=6-5]
				\arrow["{P*(J^\#B \ox c_{J^\#A,J^\#C})}", from=4-5, to=6-5]
				\arrow["{m^{J^\#A,J^\#B}_P*J^\#C}", from=7-1, to=5-1]
				\arrow["{m^{J^\#B,J^\#C}_{P*J^\#A}}"', from=7-1, to=9-1]
				\arrow["{m^{J^\#A,J^\#B \ox J^\#C}_P}", from=9-1, to=8-3]
				\arrow["{P*a_{J^\#A,J^\#B,J^\#C}}", from=6-3, to=8-3]
				\arrow["{\inv{m^{J^\#C,J^\#A}_{P*J^\#B}}}", from=5-7, to=7-7]
				\arrow["{m^{J^\#B,J^\#C}_P*J^\#A}", from=7-7, to=9-7]
				\arrow["{\inv{m^{J^\#B \ox J^\#C,J^\#A}_P}}", to=9-7, from=8-5]
				\arrow["{P*a_{J^\#B,J^\#C,J^\#A}}", from=8-5, to=6-5]
				\arrow["{P*c_{J^\#A, J^\#B \ox J^\#C}}"', from=8-3, to=8-5]
				\arrow["{(P*J^\#A)*\tilde{\gamma}_{B,C}}"', from=9-1, to=11-1]
				\arrow["{P*(J^\#A \ox \tilde{\gamma}_{B,C})}"', from=8-3, to=10-3]
				\arrow["{P*(\tilde{\gamma}_{B,C} \ox J^\#A)}", from=8-5, to=10-5]
				\arrow["{(P*\tilde{\gamma}_{B,C})*J^\#A}", from=9-7, to=11-7]
				\arrow["{P*c_{J^\#A, J^\#(B*J^\#C)}}"', from=10-3, to=10-5]
				\arrow["{m^{J^\#A,J^\#(B*J^\#C)}_P}"', from=11-1, to=10-3]
				\arrow["{\inv{m^{J^\#(B*J^\#C),J^\#A}_P}}"', from=10-5, to=11-7]
			\end{tikzcdscale}
		\end{equation}
		
		\begin{equation}\label{pfSkBr4}
			\begin{tikzcdscale}[cramped]
				&& {(P*J^\#(A*J^\#B))*J^\#C} \\
				& {(P*(J^\#A \ox J^\#B))*J^\#C} && {P*(J^\#(A*J^\#B) \ox J^\#C)} \\
				{((P*J^\#A)*J^\#B)*J^\#C} && {P*((J^\#A \ox J^\#B) \ox J^\#C)} && {P*J^\#((A*J^\#B)*J^\#C)} \\
				\\
				{(P*J^\#A)*(J^\#B \ox J^\#C)} && {P*(J^\#A \ox (J^\#B \ox J^\#C))} && {P*J^\#(A*(J^\#B \ox J^\#C))} \\
				\\
				{(P*J^\#A)*(J^\#C \ox J^\#B)} && {P*(J^\#A \ox (J^\#C \ox J^\#B))} && {P*J^\#(A*(J^\#C \ox J^\#B))} \\
				\\
				{((P*J^\#A)*J^\#C)*J^\#B} && {P*((J^\#A \ox J^\#C) \ox J^\#B)} && {P*J^\#((A*J^\#C)*J^\#B)} \\
				& {(P*(J^\#A \ox J^\#C))*J^\#B} && {P*(J^\#(A*J^\#C) \ox J^\#B)} \\
				&& {(P*J^\#(A*J^\#C))*J^\#B}
				\arrow["{m^{J^\#(A*J^\#B),J^\#C}_P}", from=1-3, to=2-4]
				\arrow["{(P*\tilde{\gamma}_{A,B})* J^\#C}", from=2-2, to=1-3]
				\arrow["{m^{J^\#A \ox J^\#B,J^\#C}_P}"', from=2-2, to=3-3]
				\arrow["{P*\tilde{\gamma}_{A*J^\#B,C}}", from=2-4, to=3-5]
				\arrow["{m^{J^\#A,J^\#B}_P*J^\#C}", from=3-1, to=2-2]
				\arrow["{m^{J^\#B,J^\#C}_{P*J^\#A}}"', from=3-1, to=5-1]
				\arrow["{P*(\tilde{\gamma}_{A,B} \ox J^\#C)}"', from=3-3, to=2-4]
				\arrow["{P*a_{J^\#A,J^\#B,J^\#C}}"', from=3-3, to=5-3]
				\arrow["{P*J^\#m^{J^\#B,J^\#C}_A}", from=3-5, to=5-5]
				\arrow["{m^{J^\#A, J^\#B \ox J^\#C}_P}", from=5-1, to=5-3]
				\arrow["{(P*J^\#A)*c_{J^\#B,J^\#C}}"', from=5-1, to=7-1]
				\arrow["{P*\varsigma_{A,J^\#B \ox J^\#C}}", from=5-3, to=5-5]
				\arrow["{P*(J^\#A \ox c_{J^\#B,J^\#C})}"', from=5-3, to=7-3]
				\arrow["{P*J^\#(A*c_{J^\#B, J^\#C})}", from=5-5, to=7-5]
				\arrow["{m^{J^\#A, J^\#C \ox J^\#B}_P}"', from=7-1, to=7-3]
				\arrow["{P*\varsigma_{A,J^\#C \ox J^\#B}}"', from=7-3, to=7-5]
				\arrow["{\inv{m^{J^\#C,J^\#B}_{P*J^\#A}}}"', to=9-1, from=7-1]
				\arrow["{m^{J^\#A,J^\#C}_P*J^\#B}"', from=9-1, to=10-2]
				\arrow["{P*a_{J^\#A,J^\#C,J^\#B}}", from=9-3, to=7-3]
				\arrow["{P*(\tilde{\gamma}_{A,C} \ox J^\#B)}", from=9-3, to=10-4]
				\arrow["{P*J^\#m^{J^\#C,J^\#B}_A}"', from=9-5, to=7-5]
				\arrow["{m^{J^\#A \ox J^\#C,J^\#B}_P}", from=10-2, to=9-3]
				\arrow["{(P*\tilde{\gamma}_{A,C})* J^\#B}"', from=10-2, to=11-3]
				\arrow["{P*\tilde{\gamma}_{A*J^\#C,B}}"', from=10-4, to=9-5]
				\arrow["{m^{J^\#(A*J^\#C),J^\#B}_P}"', from=11-3, to=10-4]
			\end{tikzcdscale}
		\end{equation}
		
	\end{landscape}
	
	\restoregeometry
	
	%\end{proof}

	\begin{remark}
		In the proposition above, we obtained a right braiding on a left skew monoidal category by considering structures with tensor $A*J^\#B$ as in (\ref{VarThm}), i.e. induced by a right action and a right adjoint. If we consider structures induced on right actegories by a left adjoint, we would get right braidings on right skew monoidal categories.
		
		Similarly, on left actegories, the induced braiding would be left -- depending on whether we use left or right adjoint to define the tensor, we would get a left braiding on either left or right skew monoidal category.
	\end{remark}

	\section{Closedness of the induced structures}\label{Closedness}
	
	In this section, we present sufficient conditions under which the structure obtained as in Theorem (\ref{TheThm}) is left/right closed.
	
	\begin{definition}
		We say that a skew monoidal structure $(\A, \sk, J)$ is \textit{right closed}, if the functor $A \sk (-)$ has a right adjoint $\inhom{A,-}$ for any $A \in \A$.
		
		Analogously, we say that $(\A, \sk, J)$ is \textit{left closed} if $(-) \sk B$ has a right adjoint $\inhom{B,-}$ for any $B \in \A$.
	\end{definition}
	
	When discussing left or right closedness for structures arising from variants of Theorem (\ref{TheThm}), it is not that important to distinguish whether we start with a left or a right action. Results for structures arising from right actions will mirror those for structures arising from left actions, if we switch the roles of left and right closedness.
	
	We will therefore only treat cases arising from left actions. On the other hand, since we are specifically interested in the existence of right adjoints of the tensor functors, it may be sensible to distinguish whether we assume that $J_*$ has a left or a right adjoint.
	
	We will observe that closedness of $(\A, \sk, J)$ is closely related to \uv{closedness} of the action of \mV, which is formalised in the following definition.
	
	\begin{definition}
		Let \mV be a (skew) monoidal category acting on a category \mA (from the left). We say that the action is \textit{left closed} (or that \mA is a \textit{left closed actegory}), if for every $A$ in \mA the functor $A_* = (-)*A: \V \rightarrow \A$ has a right adjoint.
		
		We say that the action is \textit{right closed} (or that \mA is a \textit{right closed actegory}), if for every object $X$ in \mV, the functor $X*(-): \A \rightarrow \A$ has a right adjoint $\inhom{X,-}_*$. 
	\end{definition}
	
	Let us assume $(\A, \sk, J)$ is defined with tensor $A \sk B = J_!A * B$ as in (\ref{VarThm}). Then, the functor $(-) \sk B$ can be understood as the composite functor $B_*J_!$ (where $B_* = (-)*B$). In order for this structure to be left closed, we want a right adjoint
	\[\begin{tikzcd}
		\A && \V && \A
		\arrow["{J_!}", from=1-1, to=1-3]
		\arrow["{B_*}", from=1-3, to=1-5]
		\arrow[""{name=0, anchor=center, inner sep=0}, "{\inhom{B,-}}", bend left=45, from=1-5, to=1-1]
		\arrow["\dashv"{anchor=center, rotate=-90}, draw=none, from=1-3, to=0]
	\end{tikzcd}\]
	As $J_! \dashv J_*$, it is enough to assume that \mA is a left closed actegory, i.e. that $B_*$ has a right adjoint $B_* \dashv B_\#$ for any $B$. Then, $\inhom{B,-}$ can be defined as the composite $J_*B_\#$:
	\[\begin{tikzcd}
		\A && \V && \A
		\arrow[""{name=0, anchor=center, inner sep=0}, "{J_!}", from=1-1, to=1-3]
		\arrow[""{name=1, anchor=center, inner sep=0}, "{J_*}", bend left=45, from=1-3, to=1-1]
		\arrow[""{name=2, anchor=center, inner sep=0}, "{B_*}", from=1-3, to=1-5]
		\arrow[""{name=3, anchor=center, inner sep=0}, "{B_\#}", bend left=45, from=1-5, to=1-3]
		\arrow["\dashv"{anchor=center, rotate=-90}, draw=none, from=0, to=1]
		\arrow["\dashv"{anchor=center, rotate=-90}, draw=none, from=2, to=3]
	\end{tikzcd}\]

	Now, let us turn to the case where $(\A, \sk, J)$ has tensor of the form $A \sk B = J_\#A * B$ as in (\ref{VarThm}). Then, $(-) \sk B$ can be seen again as the composite $B_*J_\#$. Under the general assumptions, neither of these functors has a right adjoint. However, if we suppose that \mA is a left closed \mV-actegory, then we have $B_* \dashv B_\#$. Therefore, to give a right adjoint $\inhom{B,-}$, it would suffice for $J_\#$ to have a right adjoint $R_J$. 
	\[\begin{tikzcd}
		\A && \V && \A
		\arrow[""{name=0, anchor=center, inner sep=0}, "{J_\#}"{description}, from=1-1, to=1-3]
		\arrow[""{name=1, anchor=center, inner sep=0}, "{J_*}"', bend right=45, from=1-3, to=1-1]
		\arrow[""{name=2, anchor=center, inner sep=0}, "{R_J}", bend left=45, from=1-3, to=1-1]
		\arrow[""{name=3, anchor=center, inner sep=0}, "{B_*}", from=1-3, to=1-5]
		\arrow[""{name=4, anchor=center, inner sep=0}, "{B_\#}", bend left=45, from=1-5, to=1-3]
		\arrow["\dashv"{anchor=center, rotate=-90}, draw=none, from=0, to=2]
		\arrow["\dashv"{anchor=center, rotate=-90}, draw=none, from=3, to=4]
		\arrow["\dashv"{anchor=center, rotate=-90}, draw=none, from=1, to=0]
	\end{tikzcd}\]
	Then, we would have $\inhom{B,-} \defeq R_JB_\#$.
	
	Right closedness for the cases discussed above amounts to the functor $J_!A*(-)$ or $J_\#A*(-)$ having a right adjoint. This would be implied if the action was right closed\footnote{It would suffice for it to be \uv{right closed on the image of $J_!$ or $J_\#$}.}. 
	
	We summarize the observations made above in the following proposition.
	
	\begin{proposition}
		\begin{enumerate}[(i)]
			\item Suppose $(\A, \sk, J)$ is a left skew monoidal category with tensor $A \sk B = J_!A * B$ as in (\ref{VarThm}). If the action is left closed, so is $(\A, \sk, J)$ with $\inhom{B,-} = J_*B_\#$.
			
			If the action is right closed, so is $(\A, \sk, J)$ with $\inhom{A,-} = \inhom{J_!A,-}_*$.
			\item Suppose $(\A, \sk, J)$ is a left skew monoidal category with tensor $A \sk B = J_\#A * B$ as in (\ref{VarThm}). If the action is left closed and in addition the functor $J_\#$ has a right adjoint $R_J$, then $(\A, \sk, J)$ is left closed with $\inhom{B,-} = R_JB_\#$.
			
			If the action is right closed, so is $(\A, \sk, J)$ with $\inhom{A,-} = \inhom{J_\#A,-}_*$.
		\end{enumerate}
	\end{proposition}

	\begin{example}
		Suppose \mC is a small category and \mD is a complete category.
		Then all pointwise right Kan extensions of functors $\C \rightarrow \D$ along functors $\C \rightarrow \D$ exist. Consider $[\C,\D]$ with left skew monoidal structure as in (\ref{exFun}), i.e. with tensor $F \sk G = \Lan_JF \circ G$ for some fixed $J$ (assuming that $\Lan_J$ exists). For any $G: \C \rightarrow \D$, the functor $G_* = (-) \circ G$ has a right adjoint given by the right Kan extension. This means this skew monoidal structure is left closed with internal hom $\inhom{G,F} = \Ran_GF \circ J$.
	\end{example}

\end{document}